\documentclass[12pt]{amsart}

\usepackage{amssymb}

\usepackage{enumerate}

\makeatletter
\@namedef{subjclassname@2010}{%
  \textup{2010} Mathematics Subject Classification}
\makeatother

\usepackage[T1]{fontenc}

\newtheorem{theorem}{Theorem}[section]
\newtheorem{corollary}[theorem]{Corollary}
\newtheorem{lemma}[theorem]{Lemma}
\newtheorem{proposition}[theorem]{Proposition}

\newtheorem{question}[theorem]{Question}
\newtheorem{conjecture}[theorem]{Conjecture}

\theoremstyle{definition}
\newtheorem{definition}[theorem]{Definition}

\newtheorem{example}[theorem]{Example}

\numberwithin{equation}{section}

\newcommand {\C}{\mathbb C}

\newcommand {\N}{\mathbb N}
\newcommand {\Z}{\mathbb Z}

\newcommand {\B}{\mathcal B}

\newcommand{\supp}{{\rm supp\,}}
\newcommand{\eps}{\varepsilon}
\newcommand{\lo}{\ell^{\,1}}

\begin{document}





\title[Finitely-generated left ideals]{Finitely-generated left ideals in Banach algebras on groups and semigroups}

\author[J.\ T.\  White]{Jared T.\ White}
\address{
Department of  Mathematics and Statistics\\
University of Lancaster\\
Lancaster LA1 4YF\\
United Kingdom}
\email{j.white6@lancaster.ac.uk}

\date{2016}

\begin{abstract}
{
Let $G$ be a locally compact group. We prove that the augmentation ideal in $L^1(G)$ is (algebraically) finitely-generated as a left ideal if and only if $G$ is finite. We then investigate weighted versions of this result, as well as a version for semigroup algebras. Weighted measure algebras are also considered. We are motivated by a recent conjecture of Dales and \.Zelazko, which states that a unital Banach algebra in which every maximal left ideal is finitely-generated is necessarily finite-dimensional. We prove that this conjecture holds for many of the algebras considered. Finally, we use the theory that we have developed to construct some examples of commutative Banach algebras that relate to a theorem of Gleason. \\
{}\\
To appear in \emph{Studia Mathematica}.
 }
\end{abstract}



\maketitle

\section{Introduction}

\noindent 
This article is concerned with finitely-generated ideals in certain Banach algebras, where ``finitely-generated'' is understood in the following sense:

\begin{definition} \label{1.1} 
Let $A$ be an algebra, and $A^\sharp$ its (conditional) unitisation. A left ideal $I$ in $A$ is \textit{finitely-generated} if there exist $n \in \N$ and $x_1,\ldots,x_n \in A$ such that $I = A^\sharp x_1 + \cdots + A^\sharp x_n$. 
\end{definition}
In the case that a left ideal is finitely-generated, it is immediate that the generators all belong to the ideal. Note also that when this definition is applied to topological algebras we do not take the closure on the right-hand side.

The Banach algebras of most interest to us will be weighted group algebras and weighted semigroup algebras, so next we define these terms precisely.

\begin{definition} \label{1.2} 
Let $S$ be a semigroup. Then a \textit{weight} on $S$ is a function $\omega : S \mapsto [1,\infty)$ such that
$$\omega(uv) \leq \omega(u) \omega(v) \quad (u, v \in S).$$
\end{definition}
\noindent
In the case where $S$ has an identity $e$, we insist that $\omega(e) = 1$. Moreover, when $G$ is a locally compact group, weights on $G$ are always assumed to be Borel functions, bounded on compact sets. 

Given a semigroup $S$, and a weight $\omega$ on $S$, we define
$$\lo(S, \omega) = \left \lbrace f : S \rightarrow \C : \Vert f \Vert_{\omega} := \sum_{u \in S} \vert f(u) \vert \omega(u) < \infty \right \rbrace.$$
The set $\lo(S, \omega)$ is a Banach space under pointwise operations with the norm given by $\Vert \cdot \Vert_\omega$, and a Banach algebra when multiplication is given by convolution. By a \textit{weighted semigroup algebra}, we shall mean a Banach algebra of this form. 

Now suppose that we have a locally compact group $G$. We shall denote left Haar measure on $G$ by $m$. Suppose that $\omega$ is a weight on $G$. Then we define
$$L^1(G, \omega) = \left \lbrace f \in L^1(G) : \Vert f \Vert_{\omega} := \int_G \vert f(t) \vert \omega(t) \, {\rm d}m(t) < \infty \right \rbrace,$$
and
$$M(G, \omega)  = \left \lbrace \mu \in M(G) : \Vert \mu \Vert_{\omega} := \int_G \omega(t) \, {\rm d} \vert \mu \vert (t) < \infty \right \rbrace.$$
The sets $L^1(G, \omega)$ and $M(G, \omega)$ are Banach algebras with respect to convolution multiplication and pointwise addition and scalar multiplication. Moreover, $L^1(G, \omega)$ is a closed subalgebra of $M(G, \omega)$. It is a Banach algebra of the form $L^1(G, \omega)$ that we refer to as a \textit{weighted group algebra}. When the group $G$ is discrete, this definition coincides with that given for a semigroup.

\begin{example} \label{1.2a} 
\begin{enumerate}
\item[(i)] The trivial weight $\omega = 1$ is always a weight on any locally compact group $G$ (or any semigroup), and in this case we recover the group algebra $L^1(G)$.
\item[(ii)] Let $G$ be a discrete group, fix a generating set  $X$, and write $\vert u \vert_X$ for the word-length of $u \in G$ with respect to $X$. Then 
$$u \mapsto (1+ \vert u \vert_X)^\alpha, \quad G \rightarrow [1, \infty), $$
defines a weight on $G$ for each $\alpha \geq 0$. We call this weight a \textit{radial polynomial weight of degree $\alpha$}.
\item[(iii)] With notation as in (ii), the map
$$u \mapsto c^{\vert u \vert_X ^\beta}, \quad G \rightarrow [1, \infty), $$
defines a weight for any $c \geq 1$ and $0 < \beta \leq 1$. We call this a \textit{radial exponential weight with base $c$ and degree $\beta$}.
\end{enumerate}
\end{example}
\noindent
More generally, a weight on a finitely-generated group $G$ is said to be \textit{radial} if there exists a generating set $X$ such that $\vert u \vert_X = \vert v \vert_X$ implies that $\omega(u) = \omega(v)$ for any $u, v \in G$. When the generating set is clear, we usually write $\vert u \vert$ in place of $\vert u \vert_X$.

Let $S$ be a semigroup, and take $\omega$ to be a weight on $S$. We define the \textit{augmentation ideal} of $\lo (S, \omega)$ to be 
$$\lo_{\,0}(S, \omega) = \left \lbrace f \in \lo(S, \omega) : \sum \limits_{u \in S} f(u) = 0 \right \rbrace.$$ 
This is the kernel of the \textit{augmentation character}, which is the map given by
$$f \mapsto \sum_{u \in S} f(u), \quad \lo(S, \omega) \rightarrow \C.$$ 
The augmentation ideal is a two-sided ideal of codimension one, and it has analogues in the weighted group algebras and the weighted measure algebras of a locally compact group $G$, also referred to as the augmentation ideals of those algebras:
$$L^1_0(G, \omega) = \left \lbrace f \in L^1(G, \omega) : \int _G f {\rm d} m =0 \right \rbrace;$$
$$M_0(G, \omega) = \left \lbrace \mu \in M(G, \omega) : \mu(G) = 0 \right \rbrace .$$
There are also corresponding augmentation characters, given by
$$f \mapsto \int_G f(t) {\rm d}m(t), \quad L^1(G, \omega) \rightarrow \C,$$
and
$$\mu \mapsto \mu(G), \quad M(G, \omega) \rightarrow \C,$$
respectively. Finally, for a semigroup $S$, we denote by $\C S$ the dense sub\-algebra of finitely-supported elements of $\lo(S)$, and define
$$\C_0 S = \lo_{\,0}(S) \cap \C S.$$

One of the central themes of this paper will be the following question:

\begin{question} \label{1.3} 
Which of the Banach algebras mentioned above have the property that the underlying group or semigroup is finite whenever the augmentation ideal is finitely-generated?
\end{question}

We now give this question some context. In 1974 Sinclair and Tullo ~\cite{st1974} proved that a left Noetherian Banach algebra, by which we mean a Banach algebra in which all the left ideals are finitely-generated in the sense of Definition \ref{1.1}, is necessarily finite dimensional. In 2012 Dales and \.Zelazko ~\cite{dz2012} conjectured the following strengthening of Sinclair and Tullo's result:

\begin{conjecture} \label{1.2} 
Let $A$ be a unital Banach algebra in which every maximal left ideal is finitely-generated. Then A is finite dimensional.
\end{conjecture}

It is this conjecture that motivates the present work. The conjecture is known to be true in the commutative case by a theorem of Ferreira and Tomassini 
\cite{ft1978}, and Dales and \.Zelazko presented a generalization of this result in their paper ~\cite{dz2012}. The conjecture is also known to be true for C*-algebras ~\cite{bk2014}, and for $\B(E)$ for many Banach spaces $E$ ~\cite{dales2013}. For instance the conjecture is known to be true when $E$ is a Banach space  which is complemented in its bidual and has a Schauder basis, or when $E = c_0 (I)$, for $I$ an arbitrary non-empty index set. However, the conjecture remains open for an arbitrary Banach space $E$.

We are interested in the conjecture for the Banach algebras arising in harmonic analysis. Our approach is to note that an affirmative answer to Question \ref{1.3} for some class of Banach algebras implies that the Dales--\.Zelazko Conjecture holds for that class. As the Dales--\.Zelazko conjecture is about unital Banach algebras, all the discrete semigroups that we consider will be monoids, in order to ensure that we are in this setting (note, however, that $\lo(S)$ can be unital without $S$ being a monoid; see for instance \cite[Example 10.15]{DLS}). However, in \S 3 we do prove some results about $L^1(G, \omega)$ for a locally compact group $G$ and a weight $\omega$, an algebra which of course is unital only when $G$ is discrete 

We now discuss our main results. Full definitions of the terminology used will be given in the body of the article. We beginning with the following answer to Question \ref{1.3} for group algebras:

\begin{corollary} \label{1.4}
Let $G$ be a locally compact group. Then $L_0^1(G)$ is finitely-generated if and only if $G$ is finite.
\end{corollary}
In particular the Dales--\.Zelazko conjecture holds for all group algebras. This result follows from Theorems \ref{3.3} and \ref{3.5}, which establish more general results. In particular, Theorem \ref{3.5} states that $M_0(G)$ is finitely-generated if and only if $G$ is compact and Theorem \ref{3.3} states that, for non-discrete $G$, $L^1(G)$ has no finitely-generated, closed, maximal left ideals at all.

The focus of \S 4 is semigroup algebras, and our main result is the following

\begin{theorem} \label{1.5} 
Let $M$ be a monoid. Then $\lo_{\,0}(M)$ is finitely-generated if and only if $M$ is pseudo-finite.
\end{theorem}

Here, ``pseudo-finite'' is a term defined in \S 4 which we deem too technical to describe here. For groups (and indeed for weakly right-cancellative monoids) pseudo-finiteness coincides with being finite in cardinality, whence the name.

We say that a sequence $(\tau_n) \subset [1, \infty)$ is \textit{tail-preserving} if, for each sequence of complex numbers $(x_n)$, we have $\sum_{n=1}^\infty \tau_n \left \vert \sum_{j=n+1}^\infty x_j \right \vert < \infty$ whenever $\sum_{n=1}^\infty \tau_n \vert x_n \vert < \infty$. This notion is explored in \S 5. In \S 6 we prove the following theorem:

\begin{theorem} \label{1.6} 
Let $G$ be an infinite, finitely-generated group, with finite, symmetric generating set $X$. Let $\omega$ be a radial weight on $G$ with respect to $X$, and write $\tau_n$ for the value that $\omega$ takes on $S_n$. Then $\lo_{\,0}(G, \omega)$ is finitely-generated if and only if $(\tau_n)$ is tail-preserving.
\end{theorem}

Here $S_n$ denotes the set group elements of word length exactly $n$ with respect to the fixed generating set $X$. This implies an affirmative answer to Question \ref{1.3} for many weighted group algebras, but also provides examples where the answer is negative:

\begin{corollary} \label{1.7} 
Let $G$ be a finitely-generated, discrete group, and let $\omega$ be a weight on $G$.
\begin{enumerate}
\item[\rm (i)]
 If $\omega$ is either a radial polynomial weight, or a radial exponential weight of degree strictly less than 1, then $\lo_{\,0}(G, \omega)$ is finitely-generated only if $G$ is finite.
\item[\rm (ii)]
If $\omega$ is a radial exponential weight of degree equal to 1, then $\lo_{\,0}(G, \omega)$ is finitely-generated.
\end{enumerate}
\end{corollary}

The proof of this corollary is given in \S 6. Finally, in \S 7, as an application of the theory developed elsewhere in the paper, we construct weights $\omega_1$ and $\omega_2$ on $\Z^+$ and $\Z$ respectively for which the Banach algebras $\ell^{\,1}(\Z^+, \omega_1)$ and $\ell^{\, 1}(\Z, \omega_2)$ fail to satisfy a converse to Gleason's Theorem on analytic structure (Theorem \ref{7.1a}). We believe that these examples illustrate new phenomena.

\section{Preliminary Results}

\noindent
We first fix some notation. We shall denote by $\Z$ the group of integers and by $\Z^+$ the semigroup of non-negative integers $\lbrace 0, 1, 2,\ldots \rbrace$. For us, $\N = \lbrace 1,2,\ldots \rbrace$.

Let $K$ be a locally compact space. We write $C_0(K)$ for the space of all complex-valued, continuous functions on $K$, which vanish at infinity, and $C_c(K)$ for the subspace of $C_0(K)$ of compactly-supported functions. We denote by $C(K)$ the linear space of all continuous functions from $K$ to $\C$. Let $\omega$ be a weight on a locally compact group $G$. Then we define $C_0(G, 1/\omega) = {\lbrace f : G \rightarrow \C : f/\omega \in C_0(G) \rbrace},$  which, as a Banach space, is isometrically isomorphic to $C_0(G)$ when given the norm $\Vert f \Vert_{1/\omega} = \Vert f/\omega \Vert_\infty$. Hence, $M(G, \omega)$ may be identified with the dual space of $C_0(G, 1/ \omega)$, where a measure $\mu \in M(G, \omega)$ acts as a bounded linear functional on $C_0(G, 1/\omega)$ via
$$f \mapsto \int_G f {\rm d} \mu.$$
By the hypothesis that $\omega$ is bounded on compact sets, we have $C_c(G) \subset L^1(G, \omega)$, and in fact $C_c(G)$ is dense in $L^1(G,\omega)$ by \cite[Lemma 1.3.5]{Kaniuth2009}.

We next prove some lemmas about arbitrary Banach algebras.

\begin{lemma} \label{2.1}
Let $A$ be a Banach algebra, let $I$ be a closed left ideal in $A$, and let $E$ be a dense subset of $I$. Suppose that $I$ is finitely-generated. Then $I$ is finitely-generated by elements of $E$.
\end{lemma}

\begin{proof}

Suppose that $I = A^\sharp x_1+ \cdots+A^\sharp x_n$, where $n \in \N$ and $x_1,\ldots,x_n \in I$. Define a map $T : (A^\sharp)^n \rightarrow I$ by 
$$ T:(a_1,\ldots,a_n) \mapsto a_1 x_1 + \cdots +a_n x_n .$$
Then $T$ is a bounded linear surjection, and, since the surjections in $\B((A^\sharp)^n, I)$ form an open set \cite[Lemma 15.3]{Stout1971}, there exists $\eps >0$ such that ${S \in \B((A^\sharp)^n, I)}$ is surjective whenever  $\Vert T - S \Vert < \eps$. Take $y_1,\ldots,y_n \in E$ with 
$$\Vert y_i - x_i \Vert < \eps/n \quad (i = 1,\ldots,n).$$ Then we see that the map $(A^\sharp)^n \rightarrow I$ defined by 
$$(a_1,\ldots,a_n) \mapsto a_1 y_1 + \cdots +a_n y_n$$
is within $\eps$ of $T$ in norm, and hence it is surjective, which implies the result.
\end{proof}

\begin{lemma}  \label{2.2} 
Let $X$ be a Banach space, with dense linear subspace $E$, and let $Y$ be a closed linear subspace of $X$ of codimension one. Then $E \cap Y$ is dense in $Y$.
\end{lemma}

\begin{proof}

Since $Y$ is a closed and codimension one subspace, $Y = \ker \varphi$ for some non-zero bounded linear functional $\varphi$. Since $Y$ is proper and closed, $E$ is not contained in $Y$. Hence there exists $x_0 \in E$ such that $\varphi (x_0) = 1$. 

Now let $y \in Y$, and take $\eps >0$. Then there exists $x \in E$ with $\Vert y - x \Vert < \eps$. Set $z = x - \varphi(x)x_0$. Then $\varphi(z) = 0$, so that $z \in E \cap Y$. Note that $\vert \varphi(x) \vert = \vert \varphi (y - x) \vert \leq \eps \Vert \varphi \Vert$, and hence $\Vert x - z \Vert = \vert \varphi(x) \vert \Vert x_0 \Vert \leq \ \eps \Vert \varphi \Vert \Vert x_0 \Vert$, so that $\Vert y-z \Vert \leq \eps \left( 1+ \Vert \varphi \Vert \Vert x_0 \Vert \right)$. Thus $\overline{E \cap Y} = Y$.
\end{proof}

\begin{lemma}  \label{2.3} 
Let $A$ be a Banach algebra, and let $B$ be a dense left ideal in  $A$. Let $I$ be a closed, maximal left ideal. Then $B \cap I$ is  dense in $I$.
\end{lemma}

\begin{proof}

As $I$ is a closed, maximal left ideal and $B$ is dense in $A$, $B$ is not contained in $I$, so that we may choose $b_0 \in B \setminus I$. Consider the left ideal $Ab_0 +I$ of $A$. As $I$ is maximal, either $Ab_0 + I = I$ or $Ab_0+I = A$. 

In the first case, we see that $ab_0 \in I$ for every $a \in A$, so that $\C b_0 +I$ is a left ideal strictly containing $I$. This forces $\C b_0 +I = A$, so that $I$ has codimension one. Therefore, in this case, the result follows from Lemma \ref{2.2}.

Hence we suppose that $Ab_0 + I = A$. Define a map $T:A \rightarrow A/I$ by $T: a \mapsto ab_0 +I$. Then $T$ is a bounded linear surjection between Banach spaces, so that, by the open mapping theorem, there exists a constant $C>0$ such that, for every $y \in A/I$, there exists $x \in A$ with $\Vert x \Vert \leq C \Vert y \Vert$ and $Tx = y$. 

Let $a \in I$ and $\eps >0$ be arbitrary. There exists $b \in B$ with $\Vert a - b \Vert < \eps$. It follows that $\Vert b+I \Vert_{A/I} \leq \eps$, so we can find $a_0 \in A$ with $\Vert a_0 \Vert \leq C \eps$ and $Ta_0 = a_0b_0 +I = b+I$. Let $c = b-a_0b_0$. Then $c \in B \cap I$, because $B$ is a left ideal, and $\Vert b - c \Vert = \Vert a_0 b_0 \Vert \leq C \eps \Vert b_0 \Vert.$ Hence  $\Vert a - c \Vert \leq  \eps (1+ C \Vert b_0 \Vert )$. As $a$ and $\eps$ were arbitrary, the result follows.
\end{proof}

\begin{corollary}  \label{2.4} 
Let $A$ be a Banach algebra with a dense, proper left ideal. Then:
\begin{enumerate}
\item[\rm (i)] $A$ has no finitely-generated, closed, maximal left ideals;
\item[\rm (ii)]$A$ has no finitely-generated, closed left ideals of finite codimension.
\end{enumerate}
\end{corollary}

\begin{proof}
(i) Assume towards contradiction that $I$ is a finitely-generated, closed, maximal left ideal in $A$. The algebra $A$ has a proper, dense left ideal $B$.Then, by Lemma \ref{2.3}, $B \cap I$ is dense in $I$, so that, by Lemma \ref{2.1}, we can find a finite set of generators for $I$ from within $B$. But then, as $B$ is a left ideal, this forces $I \subset B$, and hence $I = B$ by the maximality of $I$. But $I$ is closed, whereas $B$ is dense, and both are proper, so we have arrived at a contradiction.

\smallskip

(ii) Let $I$ be a proper, closed left ideal of finite codimension. Then $I$ is contained in some closed maximal left ideal $M$. We  may write $M = I \oplus E$, as linear spaces, for some finite-dimensional space $E \subset A$. If $I$ were finitely-generated, then the generators together with a basis for $E$ would give a finite generating set for $M$, contradicting (i). Hence $I$ cannot be finitely-generated.
\end{proof}

We note that the above corollary is of limited use since its hypothesis cannot be satisfied in a unital Banach algebra. However, in the non-unital setting it is quite effective, and we shall make use of it in \S 3. An example of a Banach algebra satisfying the hypothesis of Corollary \ref{2.4} coming from outside harmonic analysis is the algebra of approximable operators on an infinite-dimensional Banach space.

\section{The Case of a Non-Discrete Locally Compact Group}

\noindent
In this section, we shall consider Question \ref{1.3} for $L^1(G, \omega)$ and $M(G, \omega)$, where $G$ is a non-discrete, locally compact group and $\omega$ is a weight on $G$. The first result implies that, if $L^1(G) \subset C(G)$, then $G$ is discrete.

\begin{lemma} \label{3.2} 
Let $G$ be a locally compact group. Suppose that, for every precompact subset $A$ of $G$, the function $\chi_A$ is equal to a continuous function almost everywhere. Then $G$ is discrete.
\end{lemma}

\begin{proof}

Assume to the contrary that $G$ is not discrete. Then by \cite[Corollary 4.4.4]{DDLS}, or \cite[Theorem 1]{Rajagopalan}, $G$ cannot be extremely disconnected, so that there are disjoint open sets $A$ and $B$ and $x_0 \in G$ such that $x_0 \in \overline{A} \cap \overline{B}$. By intersecting with a precompact open neighbourhood of $x_0$, we may further assume that $A$ is precompact, and thus of finite measure.

Consider the function $h = \chi_A \in L^1(G)$. Then, by hypothesis, there is a continuous function $f$ and a measurable function $g$ such that $\supp g$ is an $m$-null set, with the property that $h = f+g$. In particular, $\supp g$ must have empty interior, so, for any open neighbourhood $U$ of $x_0$, we can choose $x_U \in U \cap A$ such that $x_U \notin \supp g$. Then $(x_U)$ is a net contained in $A \setminus \supp g$ converging to $x_0$. Similarly, we may find a net $(y_U)$ contained in $B \setminus \supp g$ converging to $x_0$. Then $f(x_U) = h(x_U) = 1$ for all $U$, whereas $f(y_U) = h
(y_U) = 0$ for all $U$. As both nets have the same limit, this contradicts the  continuity of $f$.
\end{proof}

\begin{theorem} \label{3.3} 
Let $G$ be a non-discrete, locally compact group, and let $\omega$ be a weight on $G$. Then $L^1(G, \omega)$ has no finitely-generated, closed, maximal left ideals, and no finitely-generated, closed left ideals of finite codimension.
\end{theorem}

\begin{proof}
Let $J = L^1(G, \omega)*C_c(G) + C_c(G)$ be the left ideal of $L^1(G, \omega)$ generated by $C_c(G)$. By \cite[Theorem 3.3.13 (i)]{D}, every element of $J$ is continuous, so that, by the previous lemma, $J$ is proper, and of course it is also dense. The result now follows from Corollary \ref{2.4}.
\end{proof}

When $G$ is a compact group, $L^2(G)$ is a Banach algebra under convolution. A trivial modification of the previous argument shows that, when $G$ is infinite and compact, $L^2(G)$ has no closed, finitely-generated maximal left ideals.

We now turn to the measure algebra. We shall exploit its duality with $C_0(G)$, so we first fix the following notation. Let $X$ be a Banach space with dual space $X'$ and let $E \subset X$ and $F \subset X'$. We write $$E^\perp = \lbrace \lambda \in  X': \langle x, \lambda \rangle = 0, \, x \in E \rbrace, \quad F_\perp = \lbrace x \in X: \langle x, \lambda \rangle = 0, \, \lambda \in F \rbrace.$$
It is well known that, for $E$ and $F$ as above, we have
\begin{equation} \label{1b}
 (F_\perp)^\perp = \overline{{\rm span}}^{w^*} F, \qquad (E^\perp)_\perp = \overline{{\rm span}} \, E.
\end{equation}

The weak*-closed left ideals of a weighted measure algebra can be characterised as follows. Analogous characterisations exist for the weak*-closed right and two-sided ideals.
\begin{lemma} \label{3.4} 
Let $G$ be a locally compact group, and let $\omega$ be  a weight on $G$. Then there is a bijective correspondence between the weak*-closed left ideals in $M(G, \omega)$ and the norm-closed subspaces of $C_0(G, 1/ \omega)$ invariant under left translation. This correspondence is given by
$$E \mapsto E^\perp,$$
for $E$ a closed subspace of $C_0(G, 1/ \omega)$ invariant under left translation. 
\end{lemma}

\begin{proof}
Let $E$ be a closed subspace of $C_0(G, 1/\omega)$, invariant under left translation. That $E^\perp$ is weak*-closed is clear. We show that it is a left ideal. Let $\mu \in E^\perp$. Then for all $f \in E$ and $y \in G$ we have
\begin{equation} \label{eq3.4} 
\int_G f(yx){\rm \, d}\mu(x) = \int_G f(x) {\rm \, d}(\delta_y*\mu)(x) = 0.
\end{equation}
Hence $\delta_y*\mu \in E^\perp$ for all $y \in G$. That $E^\perp$ is a left ideal now follows from weak*-density of the discrete measures in $M(G, \omega)$. 

Now suppose that $I$ is a weak*-closed left ideal in $M(G, \omega)$. Set $E = I_\perp$. Then, by \eqref{1b}, $E^\perp = I$. The linear subspace $E$ is clearly closed, and, for ${y \in G}, {\mu \in I}$ and $f \in C_0(G, 1/\omega)$, we have $\delta_y*\mu \in I$, so that, by \eqref{eq3.4}, $\delta_y*f \in E$. Hence $E$ is left-translation-invariant.

We have shown that the correspondence is well-defined and surjective. To see that it is injective, use \eqref{1b}.
\end{proof}

\begin{lemma} \label{0.2} 
Let $G$ be  a locally compact group. Then $M_0(G)$ is weak*-closed if and only if $G$ is compact.
\end{lemma}

\begin{proof}
If $G$ is compact, then $M_0(G) = \lbrace \text{constant functions} \rbrace ^\perp$, which is weak*-closed. 

Assume towards a contradiction that $M_0(G)$ is weak*-closed, but that $G$ is not compact. By Lemma \ref{3.4}, $E= M_0(G)_\perp$ is invariant under left translation, and using the formula $E' \cong M(G)/E^\perp = M(G) / M_0(G)$ we see that $E$ has dimension one. So there exists $f\in C_0(G)$ of norm 1 such that $E =  {\rm span} f$. There exists $x_0 \in G$ such that $\vert f(x_0) \vert = 1.$ Let $K$ be a compact subset of $G$ such that $\vert f(x) \vert < 1/2$ for all $x \in G \setminus K$. Then $K x_0^{-1} \cup x_0 K^{-1}$ is still compact, so we may choose $y \in G$ not belonging to this set, so that in particular $y x_0, y^{-1} x_0 \notin K$. Then there exists $\lambda \in \C \setminus \lbrace 0 \rbrace$ such that $\delta_y*f = \lambda f$. Hence
$$\vert f(yx_0) \vert = \vert \lambda \vert \vert f(x_0) \vert = \vert \lambda \vert < 1/2,$$
whereas
$$1= \vert f(x_0) \vert = \vert f(yy^{-1}x_0) \vert = \vert \lambda \vert \vert f(y^{-1} x_0) \vert <1/2 \cdot 1/2 = 1/4.$$
This contradiction completes the proof.
\end{proof}

The next theorem characterizes when $M_0(G)$ is finitely-generated. In particular Question \ref{1.3} has a negative answer for the measure algebra.

\begin{theorem} \label{3.5} 
Let $G$ be a locally compact group. Then $M_0(G)$ is finitely-generated as a left ideal if and only if $G$ is compact.
\end{theorem}

\begin{proof}
If $G$ is compact, and $m$ denotes the normalised Haar measure on $G$, then $\delta_e - m \in M_0(G)$ and it is easily seen, by direct computation, that $\delta_e - m$ is an identity for $M_0(G)$, so that in particular $M(G)$ is finitely-generated. 

Assume that $M_0(G)$ is finitely-generated, say 
$$M_0(G) = M(G) *\mu_1+ \cdots + M(G) * \mu_n$$
for some $ n \in \N$, and $\mu_1,\ldots,\mu_n \in M_0(G)$. Define a linear map 
$$S: M(G)^n \rightarrow M(G)$$ 
by 
$$S: (\nu_1,\ldots,\nu_n) \mapsto \nu_1*\mu_1+ \cdots +\nu_n*\mu_n.$$
As multiplication in $M(G)$ is separately weak*-continuous, $S$ is a weak*-continuous linear map, and hence $S=T^*$ for some bounded linear map $T: C_0(G) \rightarrow (C_0(G))^n$. We know that ${\rm im \,}S = M_0(G)$ is closed, implying that ${\rm im \, }T$ is closed and so we see that ${\rm im \,} S$ is weak*-closed. Hence, by Lemma \ref{0.2}, $G$ is compact.
\end{proof}

We do not know of a weighted version of this theorem, but when $G$ is discrete $M(G, \omega) = \lo(G, \omega)$, and this case will be the focus of \S 6, where it seems a very different approach is required as weak*-closure of the augmentation ideal no longer characterises finiteness of the underlying discrete group, and in particular it can happen that $\lo_{\,0}(G, \omega)$ is weak*-closed, but not finitely-generated.

Note that we have now proven  Corollary \ref{1.4}:

\begin{proof}[Proof of Corollary \ref{1.4}]
By Theorem \ref{3.3}, it is enough to consider the discrete case, which follows from Theorem \ref{3.5}.
\end{proof}

We now prove the  Dales--\.Zelazko conjecture for weighted measure algebras on non-discrete groups. We have been unable to fully resolve the discrete version, but again this is addressed in \S 6.

\begin{lemma} \label{3.5a} 
Let $G$ be a discrete group, and $\omega$ a weight on $G$. Suppose that $\lo_{\,0}(G, \omega)$ is finitely-generated as a left ideal. Then $G$ is finitely-generated.
\end{lemma}

\begin{proof}
Suppose $\lo_{\,0}(G, \omega)$ is generated by $h_1, \ldots, h_n \in \lo_{\,0}(G, \omega)$. By Lemma \ref{2.1} we may assume each $h_i$ is finitely-supported. Let $H$ be the subgroup of $G$ generated by $\bigcup \limits_{i=1}^n \supp h_i$. We show that $H = G$. Let $g \in \lo(G, \omega)$. Then, for each $i \in \lbrace 1, \ldots ,n \rbrace$,
$$\sum_{u \in H} (g*h_i)(u) = \sum_{u \in H} \sum_{st = u} g(s)h_i(t) = \sum_{s \in G} g(s) \sum_{t \in s^{-1}H} h_i(t) = 0,$$
where the final equality holds because either $s \notin H$, in which case $s^{-1}H$ is disjoint from $\supp h_i$, or else $s^{-1}H = H \supset \supp h_i$, in which case $h_i \in \lo_{\, 0}(G, \omega)$ implies that $\sum_{t \in H} h_i(t) = 0$. Since the functions $h_i$ generate $\lo_{\, 0}(G, \omega)$ it follows that $\sum_{u \in H} f(u) = 0$ for every $ f \in \lo_{\, 0}(G, \omega)$. This clearly forces $H = G$, as claimed.
\end{proof}

\begin{theorem} \label{3.6} 
The Dales--\.Zelazko conjecture holds for the algebra $M(G, \omega)$, whenever $G$ is a non-discrete locally compact group, and $\omega$ is a weight on $G$.
\end{theorem}

\begin{proof}
It is a straightforward consequence of \cite[Theorem 3.3.36(v)]{D} that $\lo(G, \omega)$ is the quotient of $M(G, \omega)$ by the closed ideal consisting of the continuous measures belonging to $M(G, \omega)$. As $G$ is non-discrete, it is uncountable, and hence, by Lemma \ref{3.5a}, $\lo_{\,0}(G, \omega)$ is not finitely-generated as  a left ideal. Taking the preimage of this ideal under the quotient map gives a codimension 1 ideal of $M(G, \omega)$, and this ideal is not finitely-generated as a left ideal.
\end{proof}

\section{The Case of a Discrete Monoid}

\noindent
We begin this section with some definitions, which generalize ideas such as word-length in group theory to the context of an arbitrary monoid. By a \textit{monoid} we mean a semigroup possessing an identity element $e$. Let $M$ be a monoid, and let $E$ be a subset of $M$. Then for $x \in M$ we define
$$E \cdot x = \lbrace ux : u \in E \rbrace, \quad
x \cdot E = \lbrace xu: u \in E \rbrace,$$
and
$$ E \cdot x^{-1} = \lbrace u \in M : ux \in E \rbrace, \quad 
x^{-1} \cdot E = \lbrace u \in M : xu \in E \rbrace.$$
We abbreviate $\lbrace u \rbrace \cdot x^{-1}$ to $u \cdot x^{-1}$, and similarly $u \cdot x$ represents the set $\lbrace u \rbrace \cdot x$. The important thing to note in these definitions is that there may not be an element $x^{-1}$, and that $u \cdot x^{-1}$ represents not an element but a set, which in general may be infinite or empty. Also, be aware that `$\cdot$' is not necessarily associative: $(x \cdot y^{-1}) \cdot z^{-1}$ is meaningful whereas $x \cdot (y^{-1}
 \cdot z^{-1})$ is not.

Now let $X \subset M$, and fix $u \in M$. We say that a finite sequence $(z_i)_{i=1}^n$ in $M$ is an \textit{ancestry for $u$ with respect to $X$} if $z_1 = u, z_n = e$, and, for each $i \in \N$ with  $1< i \leq n$, there exists $x \in X$ such that either $z_i x = z_{i-1}$ or $z_i = z_{i-1} x$.  

Denote by $H_X$ the set of elements of $M$ which have an ancestry with respect to $X$. Then
$$H_X = \lbrace e \rbrace \cup \left( \bigcup_{n \in \N, x_1,\ldots,x_n \in X} \: \bigcup_{(\eps_1,\ldots,\eps_n) \in \lbrace \pm1 \rbrace^n} \left(\ldots\left( \left(e \cdot x_1^{\eps_1} \right)\cdot x_2^{\eps_2}\right) \cdot \ldots \cdot x_n^{\eps_n} \right) \right). $$
We say that the monoid $M$ is \textit{pseudo-generated} by $X$ if $M= H_X$; this is the same notion as what is termed being \textit{right unitarily generated by $X$} in \cite{Kobayashi2007}.
Observe that when $M$ is not just a monoid but a group, $M$ is pseudo-generated by $X$ if and only if it is generated by $X$. We say that $M$ is \textit{finitely pseudo-generated} if $M$ is pseudo-generated by some finite set $X$. 

Given a subset $X$ of $M$ we set $B_0 = \lbrace e \rbrace$ and for each $n \in \N$ we set
$$B_n = \lbrace e \rbrace \cup \left( \bigcup_{x_1,\ldots,x_k \in X, \: k \leq n} \: \bigcup_{(\eps_1,\ldots,\eps_k) \in \lbrace \pm 1 \rbrace^k} \left(\ldots\left( \left(e \cdot x_1^{\eps_1} \right) \cdot x_2^{\eps_2}\right) \cdot \ldots \cdot x_k^{\eps_k} \right)\right) $$
and 
\begin{equation} \label{S} 
S_n = B_n \setminus B_{n-1}.
\end{equation}
The set $B_n$ consists of those $u$ in $M$ which have an ancestry of length at most $n$ with respect to $X$. Of course the  sets $B_n$ and $S_n$ depend on  $X$, but we suppress this in the notation as $X$ is usually clear from the context. Finally, we say that $M$ is \textit{pseudo-finite} if there is some $n \in \N$ and a finite subset $X$ of $M$ such that every element of $M$ has an ancestry with respect to $X$ of length at most $n$, or equivalently if
$$M = \bigcup_{k=0}^n B_k .$$
Again, for a group $M$, $M$ pseudo-finite if and only if it is finite. 

To see an example of a monoid which is pseudo-finite, but not finite, take any infinite monoid $M$ and add a zero $\theta$ to obtain $M^0 = M \cup \lbrace \theta \rbrace $. Then 
$$M^0 = \theta \cdot \theta^{-1},$$
so that $M^0$ is pseudo-finite. Incidentally, this also furnishes us with an example where associativity of `$\cdot$' fails, even though all expressions involved are meaningful: we have $(\theta \cdot \theta^{-1}) \cdot e = M^0$, whereas $\theta \cdot (\theta^{-1} \cdot e) = \emptyset$.

In the next two lemmas we establish a version of Lemma \ref{3.5a} for monoids.

\begin{lemma} \label{4.2}
Let $M$ be a monoid, and let $X \subset M$. Then we have 
$${H_X \cdot u}, {H_X \cdot u^{-1}} \subset H_X \quad (u \in H_X).$$ 
\end{lemma}

\begin{proof}
To see this, we define $H_0 = \lbrace e \rbrace \cup X$, and subsequently
$$H_k = \left(\bigcup \limits_{x \in X} H_{k-1} \cdot x \right) \cup \left( \bigcup \limits_{x \in X} H_{k-1} \cdot x^{-1} \right)$$
for $k \in \N$. It is easily seen that 
$$H_X = \bigcup \limits_{k=0}^\infty H_k.$$

We establish the lemma by induction on $k$ such that $u \in H_k$. The case $k=0$ follows just from the definition of $H_X$. So suppose that $k >0$. Then either $u = zx$ or $ux = z$ for some $z \in H_{k-1}$ and $x \in X$. Consider the first case, and let $h \in H_X$. Then $hu = hzx$. By the induction hypothesis $hz \in H_X$, and hence $hu = hzx \in H_X$ by the case $k=0$. Similarly, if $y \in M$ is such that $yu = yzx \in H_X$, then $yz \in H_X$, and so $y \in H_X$ by the induction hypothesis applied to $z$.

Similar considerations apply in the case where $u$ has the property that $ux = z$ for some $z \in H_{k-1}$ and some $x \in X$, and we see that in either case $H_X \cdot u$, $H_X \cdot u^{-1} \subset H_X$, completing the induction.
\end{proof}

\begin{lemma} \label{4.3} 
Let $M$ be a monoid, let $\omega$ be a weight on $M$, and suppose that $\lo_{\,0}(M, \omega)$ is finitely-generated as a left ideal in $\lo(M, \omega)$. Then $M$ is finitely pseudo-generated.
\end{lemma}

\begin{proof}
Write $A = \lo(M, \omega)$. Since $\C_0M$ is dense in  $\lo_{\,0}(M, \omega)$, by Lemma \ref{2.1} we may suppose that 

\begin{equation} \label{gens} 
\lo_{\,0}(M, \omega) = A*h_1 + \cdots + A*h_n \
\end{equation}
for some $h_1,\ldots,h_n \in \C_0 M$. Set
 $$X = \bigcup \limits_{i =1}^{n} \supp h_i,$$
so that $X$ is a finite set. We shall complete the proof by showing that $X$ pseudo-generates $M$.

Write $H = H_X$. We observe that, for $s \in M$, if $s^{-1} \cdot H \cap H \neq \emptyset$, then $s \in H$. Indeed, suppose that $u \in s^{-1} \cdot H \cap H$. Then $su \in H$, and hence $s \in H \cdot u^{-1}$, which is a subset of $H$ by Lemma \ref{4.2}. 

Now let $g \in A$ be arbitrary. Then, for every $i \in \lbrace 1,\ldots, n \rbrace$, we have

\begin{align*}
\sum \limits_{u \in H} (g*h_i)(u) &= \sum \limits_{u \in H} \sum \limits_{st = u} g(s)h_i(t) 
 = \sum \limits_{u \in H} \sum \limits_{s \in M} \sum \limits_{t \in s^{-1} \cdot u} g(s) h_i(t) \\
&= \sum \limits_{s \in M} \left( g(s) \sum \limits_{t \in s^{-1} \cdot H} h_i(t) \right) 
= \sum \limits_{s \in H} \left( g(s) \sum \limits_{t \in s^{-1} \cdot H} h_i(t) \right),
\end{align*}
where the last equality holds because $s^{-1} \cdot H \cap \supp h_i \subset s^{-1} \cdot H \cap H = \emptyset$ whenever $s \notin H$. However, when $s \in H$, then, for every $x \in  \supp h_i$, we have $sx \in H$ by Lemma \ref{4.2}, which implies that $\supp h_i \subset s^{-1} \cdot H$. It follows that 
$$ \sum \limits_{t \in s^{-1} \cdot H} h_i(t) = 0$$
because $h_i \in \C_0M$. Hence
$$\sum \limits_{u \in H} (g*h_i)(u) = 0.$$
By (\ref{gens}), this implies that 
$$ \sum \limits_{u \in H} f(u) = 0$$
 for every  $ f \in \lo_{\,0}(M)$. But this clearly forces $M = H$, as required.
\end{proof}

Suppose that a monoid $M$ is pseudo-generated by a finite set $X$. 
Given $f \in \lo(M)$, we define a sequence of scalars  $(\sigma_n(f))$ by
$$\sigma_n(f) = \sum \limits_{u \in B_n} f(u).$$

\begin{lemma} \label{4.4} 
Let $M$ be a monoid and $X \subset M$. Let the sets $B_n$ in the definition of $\sigma_n$ refer to $X$. Then, for every $g \in \lo(M)$ and every $x \in  X$ we have
$$\sum_{n=1}^\infty \vert \sigma_n(g*(\delta_e - \delta_x)) \vert < \infty.$$
\end{lemma}

\begin{proof}
Write $\sigma_n =\sigma_n(g*(\delta_e-\delta_x))$.
Since 
$$g*(\delta_e-\delta_x) = \sum \limits_{u \in M} g(u) \delta_u- g(u) \delta_{ux},$$
it follows that
$$\sigma_n = \sum \limits_{u \in B_n} g(u) - \sum \limits_{v \in B_n \cdot x^{-1}}g(v).$$ 
If $u \in B_{n-1}$, then $ux \in B_n$, implying that $B_{n-1} \subset B_n \cap B_n \cdot x^{-1}$. Hence
\begin{align*}
\sigma_n &= \sum_{u \in B_n \setminus B_{n-1}} g(u) + \sum_{u \in B_{n-1}}g(u) - \left( \sum_{u \in B_n \cdot x^{-1} \setminus B_{n-1}} g(u) + \sum_{u \in B_{n-1}} g(u) \right) \\
&= \sum \limits_{u \in S_n}g(u) - \sum \limits_{v \in B_n \cdot x^{-1} \setminus B_{n-1}} g(v).
\end{align*}
Notice that  $B_n \cdot x^{-1} \subseteq B_{n+1}$, so that 
$$B_n \cdot x^{-1} \setminus B_{n-1} \subseteq B_{n+1} \setminus B_{n-1} = S_n \cup S_{n+1}.$$
 Hence
\begin{align*}
\vert \sigma_n \vert \leq \sum \limits_{u \in S_n} \vert g(u) \vert + \sum \limits_{u \in S_n \cup S_{n+1}} \vert g(u) \vert 
= 2 \sum \limits_{u \in S_n} \vert g(u) \vert + \sum \limits_{u \in S_{n+1}} \vert g(u) \vert,
\end{align*}
so that $\sum \limits_{n=1}^\infty  \vert \sigma_n \vert \leq 3 \sum \limits_{u \in M} \vert g(u) \vert < \infty$, using the fact that the sets $S_n$ are pairwise disjoint.
\end{proof}

We shall now prove Theorem \ref{1.5} in the next two propositions.

\begin{proposition} \label{4.5}
Let $M$ be a monoid such that $\lo_{\,0}(M)$ is finitely-generated as a left ideal. Then $M$ is pseudo-finite.
\end{proposition}

\begin{proof}
By Lemmas \ref{2.1} and \ref{2.2}, $\lo_{\,0}(M)$ is generated by finitely many elements of $\C_0 M$. Suppose that $h = \sum \limits_{i =1}^N \alpha_i \delta_{u_i}$ is one of these generators, where $N \in \N$ and $u_1, \ldots ,u_N \in M$. Then a simple calculation exploiting  the fact that $\sum_{i=1}^N \alpha_i = 0$ shows that 
$$h = \sum \limits_{i=1}^{N-1} \left( \sum \limits_{j = 1}^i \alpha_j \right)(\delta_{u_i} - \delta_{u_{i+1}}).$$
Writing $\delta_{u_i} - \delta_{u_{i+1}} = (\delta_{e}-\delta_{u_{i+1}})-(\delta_{e}-\delta_{u_i})$ shows that
$$h = \sum_{i=1}^N   \beta_i (\delta_e - \delta_{u_i})$$
for some $\beta_1,\ldots,\beta_N \in \C$. It follows that there is some finite subset $Y$ of $M$ such that $\lo(M)$ is generated by elements of the form $\delta_e - \delta_u \; (u \in Y)$.

By Lemma \ref{4.3}, $M$ is pseudo-generated by some finite set $X$. Enlarging $X$ if necessary, we may suppose that $Y \subset X$. It then follows from Lemma \ref{4.4} that $(\sigma_n(f)) \in \lo(\N)$ for every $f \in \lo_{\,0}(M)$, since now every element of $\lo_{\,0}(M)$ is a linear combination of elements of the form considered in that lemma.  We now show that this gives a contradiction in the case where $M$ is not pseudo-finite by constructing an element $f$ of $\lo_{\,0}(M)$ for which $(\sigma_n(f)) \notin \lo(\N)$.

Assume that $M$ is not pseudo-finite. Then no $B_n$ is the whole of $M$, but, by the definition of $X$, $\bigcup \limits_{n = 1}^{\infty} B_n = M$, so there exists an increasing sequence $(n_k)$ of natural numbers such that $B_{n_{k-1}} \varsubsetneq B_{n_k}$ for every $k \in \N$. Select ${u_{n_k} \in B_{n_k} \setminus B_{n_{k-1}} \: (k \in \N)}$. Let $ \zeta = \sum \limits_{j = 1}^{\infty} 1/j^2$, and define $f \in \lo_{\, 0}(M)$ by $f(e) = \zeta, {f(u_{n_k})  = -1/k^2}$ and $f(u) = 0$ otherwise.
Then 
$$\sigma_{n_k}(f) = \zeta  - \sum \limits_{j = 1}^{k} \dfrac{1}{j^2} = \sum \limits_{j = k +1}^{\infty} \dfrac{1}{j^2} \geq \frac{1}{k} \quad (k \in \N).$$
Hence $ \sum \limits_{k = 1}^{\infty} \vert \sigma_{n_k}(f) \vert = \infty$, so that $(\sigma_n(f)) \notin \lo(\N)$, as required.
\end{proof}

The converse of Proposition \ref{4.5} is also true, as we shall now prove, completing the proof of Theorem \ref{1.5}.

\begin{proposition} \label{4.6}
Let $M$ be a pseudo-finite monoid. Then $\lo_{\,0}(M)$ is finitely-generated.
\end{proposition}

\begin{proof}
Let $X = \lbrace x_1,\ldots, x_r \rbrace$ be a finite pseudo-generating set for $M$ such that $B_n = M$ for some $n \in \N$. For $k \in \N$, we define
$$\Lambda_k = \lbrace f \in \lo_{\,0}(M) : \supp f \subset B_k \rbrace,$$
and use induction on $k$ to show that $\Lambda_k$ is contained in a finitely-generated ideal which is contained in $\lo_{\,0}(M)$.

Write $A = \lo(M)$, and denote  the  augmentation character on $\lo(M)$ by $\varphi_0$. For $f \in \Lambda_1$, we may write

\begin{align*}
f &= f(e) \delta_e + \sum \limits_{i = 1}^r f(x_i) \delta_{x_i} \\
&= f(e)(\delta_e - \delta_{x_1}) + (f(e) + f(x_1))(\delta_{x_1} - \delta_{x_2})+ \\
&\quad \: \cdots +(f(e)+ \cdots + f(x_{r-1}))(\delta_{x_{r-1}} - \delta_{x_{r}}).
\end{align*}
It follows that $\Lambda_1 \subset A*(\delta_e - \delta_{x_1}) + \cdots + A*(\delta_{x_{r-1}} - \delta_{x_{r}}).$ This establishes the base case.

Consider $k>1$. By the induction hypothesis, there exist $m \in \N$ and $p_1, \ldots,p_m \in \lo_{\,0}(M)$ such that
$$\Lambda_{k-1} \subset A*p_1+ \cdots + A*p_m.$$
Write $B_k$ as
$$B_k = \lbrace e \rbrace \cup \left( \bigcup_{i=1}^r B_{k-1} \cdot x_i \right) \cup \left(\bigcup_{i=1}^r B_{k-1} \cdot x_i^{-1} \right).$$
Write $f \in \Lambda_k$ as
$$f = f(e) \delta_e+g_1 + \cdots + g_r + h_1 + \cdots + h_r,$$
where $\supp g_i \subset B_{k-1} \cdot x_i$ and $\supp h_i \subset B_{k-1} \cdot x_i^{-1}$. 
Then 
\begin{align*}
f = \sum \limits_{i = 1}^r (g_i - \varphi_0(g_i) \delta_{x_i}) &+ \sum \limits_{i = 1}^r (h_i - \varphi_0(h_i) \delta_e) \\
&+ \sum \limits_{i=1}^r \varphi_0(g_i)\delta_{x_i} +\left(f(e) + \sum \limits_{i=1}^r \varphi_0 (h_i) \right) \delta_e.
\end{align*}
We note that 
$$ \sum \limits_{i=1}^r \varphi_0(g_i)\delta_{x_i} +\left(f(e) + \sum \limits_{i=1}^r \varphi_0 (h_i) \right) \delta_e \in  A*(\delta_e - \delta_{x_1}) + \cdots + A*(\delta_{x_{r-1}} - \delta_{x_{r}})$$
by the base case. 
Fix $i \in \lbrace 1, \ldots, r  \rbrace$. Each $ u \in B_{k-1} \cdot x_i$ can be written $u =u'x_i$ for some $u' \in B_{k-1}$ (which depends on $u$, and may not be unique), and we calculate that
$$g_i = \sum \limits_{u \in B_{k-1} \cdot x_i} g_i(u) *\delta_{u'x_i}  = g_i' *\delta_{x_i},$$
where $g_i' = \sum \limits_{u \in B_{k-1} \cdot x_i} g_i(u) \delta_{u'}$. Moreover,
$$g_i - \varphi_0(g_i) \delta_{x_i} = (g_i' - \varphi_0(g_i) \delta_e)*\delta_{x_i}.$$
The support of $g_i' - \varphi_0(g_i) \delta_e$ is contained in $B_{k-1}$, and so, by the induction hypothesis, we have
$$g_i' - \varphi_0(g_i) \delta_e \in A*p_1+ \cdots +A*p_m,$$
whence
$$g_i - \varphi_0(g_i) \delta_{x_i} \in A*p_1 *\delta_{x_i} + \cdots +A*p_m *\delta_{x_i}.$$

Now consider $h_i - \varphi_0(h_1) \delta_e$. We have
$$h_i *\delta_{x_i} = \sum \limits_{u \in B_{k-1} \cdot x_i^{-1}} h_i(u) \delta_{ux_i},$$
so that $\supp (h_i *\delta_{x_i}) \subset B_{k-1}$ and, in particular, $\supp (h_i *\delta_{x_i} - \varphi_0(h_i) \delta_{x_i}) \subset B_{k-1}$ (as $k \geq 2$). It then follows from the induction hypothesis that 
$$(h_i - \varphi_0(h_i) \delta_e ) *\delta_{x_i} = a_1*p_1 + \cdots + a_m*p_m$$
for some $a_1,\ldots,a_m \in A$. So
\begin{align*}
h_i-\varphi_0(h_i)\delta_e &= (h_i-\varphi_0(h_i)\delta_e)*(\delta_e - \delta_{x_i}) +a_1*p_1+ \cdots +a_m*p_m  \\
&\in A*(\delta_e - \delta_{x_i})+ A*p_1+ \cdots +A*p_m.
\end{align*}

We now conclude that
$$\Lambda_k \subset \sum \limits_{i=1}^m A*p_i + \sum \limits_{i,j } A*p_i *\delta_{x_j} + \sum \limits_{i=1}^r A*(\delta_e - \delta_{x_i}) +\sum \limits_{i=1}^{r-1} A*(\delta_{x_i}- \delta_{x_{i+1}}).$$
This completes the induction. When $k = n$, we obtain the theorem.
\end{proof}

We recall the following standard definitions:
\begin{definition} \label{4.1a} 
Let $M$ be a monoid. Then:
\begin{enumerate}
\item[(i)] $M$ is \textit{right cancellative} if $a=b$ whenever $ax=bx$ $(a, b, x \in M)$;
\item[(ii)] $M$ is \textit{weakly right cancellative} if, for every $a, x \in M$, the set $a \cdot x^{-1}$ is finite.
\end{enumerate}
\end{definition}

It is easily seen from the definitions that a weakly right cancellative monoid is pseudo-finite if and only if it is finite. Hence, Question \ref{1.3} and the Dales--\.Zelazko conjecture both have answers in the affirmative for the class of Banach algebras of the  form $\lo(M)$, where $M$ is a weakly right cancellative monoid. However, it remains open whether the Dales--\.Zelazko conjecture holds for $\lo(M)$ for an arbitrary monoid $M$.

\section{$\tau$-Summable Sequences}
\noindent
In this section $\tau = (\tau_n)$ will always be a sequence of real numbers, all at least 1. We say that a sequence of complex numbers $(x_n)$ \textit{$\tau$-summable} if
$$ \sum \limits_{n=1}^{\infty} \tau_n \vert x_n \vert < \infty.$$  
Note that if $(x_n)$ is $\tau$-summable for some $\tau$, then in particular $(x_n) \in \lo$. 

We say that $\tau$ is \textit{tail-preserving} if the sequence $\left( \sum \limits_{j=n+1}^{\infty} x_j \right)$ is $\tau$-summable whenever $(x_n)$ is $\tau$-summable.
For example,  the constant 1 sequence is not tail-preserving (as can be seen by considering, for instance, the sequence $x_n = 1/n^2 \: (n   \in \N)$), but it will be a consequence of Proposition \ref{5.2}, below, that $\tau_n = c^n$ is tail-preserving for each $c>1$. The main result of this section is an intrinsic characterization of tail-preserving sequences, given in Proposition \ref{5.2}. The results of this section will underlie our main line of attack when we consider questions involving weights on discrete groups in \S 6 and \S 7.

Our approach is to consider the Banach  spaces $\lo(\tau)$,  defined by
$$\lo(\tau) = \left \lbrace (x_n) \in \C^{\N} : \sum \limits_{n=1}^{\infty} \tau_n \vert x_n \vert < \infty \right \rbrace, $$
with the norm given by
$$ \Vert (x_n) \Vert_\tau  = \sum \limits_{n=1}^{\infty} \tau_n \vert x_n \vert,$$
so that $\lo(\tau)$ is exactly the set of $\tau$-summable sequences. Each space $\lo(\tau)$ is in fact isometrically isomorphic to $\lo$.

In the next proposition we denote by $c_{00}$ the space of finitely-supported complex sequences, and we write $c_{00}^+$ for the set of those sequences in $c_{00}$ whose terms are all non-negative reals.

\begin{lemma} \label{5.1} 
Let $\tau = (\tau_n)$ be a sequence in $[1, \infty)$. Then the following are equivalent:
\begin{enumerate}
\item[\rm (a)] $\tau$ is not tail-preserving;
\item[\rm (b)] the set $\left \lbrace x \in \lo(\tau): \left( \sum \limits_{j=n+1}^{\infty} x_j \right) \in \lo(\tau) \right \rbrace$ is meagre in $\lo(\tau)$;
\item[\rm (c)] there is a sequence of vectors $(x^{(k)})$ in $c_{00}^+$ such that 
$$\Vert x^{(k)} \Vert_\tau \leq 1 \: (k \in \N) \quad {\rm and} \quad \lim \limits_{k \rightarrow \infty} \sum \limits_{n=1}^{\infty} \tau_n \left \vert \sum \limits_{j=n+1}^{\infty} x_j^{(k)} \right \vert = \infty .$$
\end{enumerate} 
\end{lemma}

\begin{proof}
For $x = (x_j) \in \lo(\tau)$, write
$$T(x) = \sum \limits_{n=1}^{\infty} \tau_n \left \vert \sum \limits_{j=n+1}^{\infty} x_j \right \vert,$$
which may take the value infinity. 
Define sets
$$E_m = \lbrace x \in \lo(\tau) : T(x) \leq m \rbrace \quad (m \in \N), \quad {\rm and} \quad E = \bigcup \limits_{m = 1}^{\infty} E_m .$$ 

Fix $m \in \N$. We claim that $E_m$ is closed in $(\lo(\tau), \Vert \cdot \Vert_\tau)$. For this, suppose $(x^{(i)})$ is a sequence in $E_m$ which converges to some point $y \in \lo(\tau)$. Then for each $p, i \in \N$, we have
\begin{align*}
\sum_{n=1}^{p} \tau_n \left \vert \sum_{j=n+1}^{\infty} y_j \right \vert &\leq \sum \limits_{n=1}^{p} \tau_n \left( \left \vert \sum _{j=n+1}^{\infty} \left(y_j- x_j^{(i)} \right) \right \vert + \left \vert \sum _{j=n+1}^{\infty} x_j^{(i)} \right \vert \right) \\
&\leq   \sum _{n=1}^{p} \tau_n \sum \limits_{j=n+1}^\infty \vert y_j - x_j^{(i)}  \vert +m \\
&\leq p \left ( \max \limits_{k = 1,\ldots,p} \tau_k \right) \sum _{j=1}^{\infty} \vert y_j - x_j^{(i)} \vert +m \\
&\leq p \left ( \max \limits_{k = 1,\ldots,p} \tau_k \right) \Vert y - x^{(i)} \Vert_\tau +m.
\end{align*}
Letting $i \rightarrow \infty$, we obtain
$$\sum \limits_{n=1}^{p} \tau_n \left \vert \sum \limits_{n+1}^{\infty} y_j \right \vert \leq m.$$
As $p$ was arbitrary, this implies that $T(y) \leq m$, and so $y \in E_m$, giving the claim.

We first prove that (c) implies (b). Take $(x^{(k)})$ as in (c). Let $m \in \N$ and $\eps >0$. Then we can find $k \in \N$ such that $T(x^{(k)}) \geq \dfrac{1}{\eps}(2m+1)$. Let $y = \eps x^{(k)}$. Then $\Vert y \Vert_\tau \leq \eps$, but, for each $x \in E_m$, we have
$$T(x + y) \geq T(y) - T(x) \geq (2m+1) - m = m+1,$$
showing that $ x+y \notin E_m$. It follows that each $E_m$ has empty interior, and hence $E$ is meagre, which is exactly the statement in (b). 

By the Baire Category Theorem, (b) implies (a).

Suppose (a) holds and that $x = (x_n)$ is a sequence in $\ell^{\, 1} (\tau)$ such that $T(x) = \infty$. Since replacing $(x_n)$ by $( \vert x_n \vert )$ only increases $T(x)$ whilst preserving the norm, we may assume that each $x_n$ is a non-negative real number. By scaling, we may assume that $\Vert x \Vert_\tau \leq 1$. Given $k \in \N$, choose $N \in \N$ such that 
$$\sum \limits_{n=1}^N \tau_n \left( \sum \limits_{j = n+1}^{\infty} x_j \right) \geq k.$$ Take an integer $M > N$ and such that 
$$\sum \limits_{j = M+1}^\infty x_j \leq \left( \sum \limits_{n = 1}^{N} \tau_n \right)^{-1}. $$
Define $x^{(k)} \in c_{00}$ by\\
$$x^{(k)}_j = \begin{cases}
x_j & \mbox{if } j \leq M  \\
  0 & \mbox{otherwise.}
\end{cases}$$
Then
\begin{align*}
T(x^{(k)}) &= \sum_{n=1}^\infty  \tau_n \left( \sum_{j=n+1}^{M} x_j  \right)
\geq \sum \limits_{n=1}^N \tau_n \left( \sum_{j = n+1}^\infty x_j - \sum_{j = M+1}^\infty x_j \right) \\
&= \sum \limits_{n=1}^N \tau_n \left( \sum_{j = n+1}^\infty x_j \right) - \sum_{n=1}^N \tau_n \left( \sum_{j = M+1}^\infty x_j \right)
\geq k-1.
\end{align*}
Hence $T(x^{(k)}) \rightarrow \infty$ as $k \rightarrow \infty$, whilst $\Vert x^{(k)} \Vert_\tau \leq \Vert x \Vert_\tau \leq 1$, so that (c) holds.
\end{proof}

We are now able to give our intrinsic characterization of tail-preserving sequences.

\begin{proposition} \label{5.2}
Let $\tau = (\tau_n)$ be a sequence in $[1, \infty)$. Then the following are equivalent:
\begin{enumerate}
\item[\rm (a)] $\tau$ is tail-preserving;
\item[\rm (b)] there exists a constant $D>0$ such that
\begin{equation} \label{eq5.1} 
\tau_{n+1} \geq D \sum_{j=1}^n \tau_j \quad (n \in \N);
\end{equation}
\item[\rm (c)] $\liminf \limits_n \left(\tau_{n+1} \Big/ \sum \limits_{i=1}^{n} \tau_i \right) >0$.
\end{enumerate}
\end{proposition}

\begin{proof}
The equivalence of (b) and (c) is clear. We show the equivalence of (a) and (b). 

Suppose that (b) holds. Assume towards a contradiction $\tau$ is not tail-preserving. Then there exists a sequence $(x^{(k)}) \subset c_{00}^+$ as in Lemma \ref{5.1}(c).  Let $s_k = \max \supp x^{(k)}$. Let $e^{(j)}$ denote the sequence which has a one in the $j$th place and is zero elsewhere and let $T$ be defined as in the proof of Lemma \ref{5.1}. Then
$$T (x^{(k)}) = \sum_{j = 1}^{s_k} x^{(k)}_j T(e^{(j)}) = \sum_{j = 1}^{s_k}x^{(k)}_j \left( \sum_{n=1}^{j-1} \tau_n \right),$$
where the first equality holds because the terms of $x^{(k)}$ are presumed non-negative. 
Take $k \in \N$.  Then
\begin{align*}
\Vert x^{(k)} \Vert_\tau = \sum \limits_{j =1}^{s_k}x^{(k)}_j \tau_j 
\geq D  \sum \limits_{j =1}^{s_k} x^{(k)}_j \left( \sum \limits_{i=1}^{j-1} \tau_i \right) 
= D T(x^{(k)}).
\end{align*}
This gives a contradiction because the final term tends to infinity with $k$, whereas $\Vert x^{(k)} \Vert_\tau \leq 1$ for all $k \in \N$. Hence we have established that (b) implies (a).

Now suppose that (b) does not hold. Then there exists a strictly increasing sequence $(n_k)$ in $\N$ such that 
$$\lim \limits_k \left( \tau_{{n_k}+1} \bigg/ \sum \limits_{i=1}^{n_k} \tau_i \right) = 0.$$ 
Define $x^{(k)} = \dfrac{1}{\tau_{{n_k}+1}} e^{({n_k}+1)} \: (k \in \N)$, so that $\Vert x^{(k)} \Vert_\tau = 1$. Then 
$$T(x^{(k)}) = \dfrac{1}{\tau_{{n_k}+1}} \sum \limits_{i=1}^{n_k +1} \tau_i \quad (k \in \N),$$
which tends to infinity as $k$ tends to infinity. Hence clause (c) of Lemma \ref{5.1} is satisfied, and this concludes the proof.
\end{proof}

As we remarked above, it is an immediate consequence of this proposition that the sequence $(c^n)$ is tail-preserving for each $c>1$

The following lemma concerns the growth of tail-preserving sequences. Part (ii) implies that, if $(\tau_n)$ is tail-preserving and $\tau_n' \geq \tau_n$ for all $n$, then $(\tau_n')$ is not necessarily tail-preserving.

\begin{lemma} \label{5.3} 
\begin{enumerate}
\item[\rm (i)] Let $\tau = (\tau_n)$ be a tail-preserving sequence, and let $D>0$ satisfy \eqref{eq5.1}. Then
$$\tau_{j+1} \geq D(D+1)^{j-1}\tau_1 \quad (j \in \N).$$
\item[\rm (ii)] Let $\rho >1$. There exists a sequence $(\tau_n) \subset [1, \infty)$ such that $\rho^n \leq \tau_n$ for all $n \in \N$, but $(\tau_n)$ is not tail-preserving.
\end{enumerate}
\end{lemma}

\begin{proof}
(i) We proceed by induction on $j \in \N$. The case $j = 1$ is immediate from \eqref{eq5.1}. Now suppose that $j>1$, and assume that the result holds for all $i< j$. Then we have
\begin{align*}
\tau_{j+1} &\geq D \sum \limits_{i=1}^{j} \tau_i 
\geq D[ D(D+1)^{j-2} + \cdots + D(D+1) +D +1]\tau_1 \\
&= D(D+1)^{j-1} \tau_1.
\end{align*}
Hence the result also holds for $j$.

(ii) Define integers $n_k$ recursively by $n_1=1$ and $n_k = n_{k-1}+k+1$ for $k \geq 2$. Then define
$$\tau_j = \rho^{n_k+1} \quad (n_{k-1}+1 < j \leq n_k +1).$$
Then clearly $\tau_j \geq \rho^j$ for all $j \in \N$, and
$$\frac{\tau_{n_k+1}}{\sum_{j=1}^{n_k}\tau_j} \leq \frac{\tau_{n_k+1}}{\sum_{j=n_{k-1}+2}^{n_k}\tau_j} = \frac{1}{k} \rightarrow 0.$$
Hence $(\tau_n)$ violates condition (c) of Proposition \ref{5.2}, so cannot be tail-preserving.
\end{proof}

\section{Weighted Discrete Groups}
\noindent
In this section $G$ will denote a discrete group, with finite generating set $X$, and $\omega$ will be a weight on $G$. Without loss of generality we may suppose that $X$ is symmetric (we recall that a subset $X$ of a group $G$ is \textit{symmetric} if $X = X^{-1}$). We shall consider whether $\lo(G, \omega)$ is finitely-generated. We note that when considering Question \ref{1.3} and Conjecture \ref{1.2} for $L^1(G, \omega)$, Theorem \ref{3.3} and Lemma \ref{4.3} allow us to reduce to this setting. As noted at the  end of Section 3, similar remarks pertain to $M(G, \omega)$. We define a sequence of real numbers, all at least 1, by
\begin{equation} \label{eq6.1} 
\tau_n = \min \limits_{u \in S_n}  \omega(u),
\end{equation}
where $S_n$ is defined by \eqref{S}. As we are now in the group setting, $S_n$ is exactly the the set of group elements of word-length $n$ with respect to $X$.
We write 
\begin{equation} \label{eq6.2} 
C = \max \limits_{x \in X} \omega(x).
\end{equation}

\begin{lemma} \label{6.1} 
With $\tau_n \; (n \in N)$ and $C$ defined by \eqref{eq6.1} and \eqref{eq6.2}, respectively, we have $ \tau_n \leq C\tau_{n+1}$ for all $n \in \N$.
\end{lemma}

\begin{proof}
For each $n \in \N$, take $y_n \in S_n$ satisfying $\omega(y_n) = \tau_n$. Then $y_{n+1} = zx$ for some $z \in S_n$ and some $x \in X$, so 
\begin{align*}
\tau_n = \omega(y_n) \leq  \omega(z)  
 = \omega(y_{n+1}x^{-1}) 
 \leq  C \omega(y_{n+1}) = C \tau_{n+1},
\end{align*}
giving the result.
\end{proof}

In the next lemma, notice that parts (i) and (ii) depend on the weight having the specified properties, whereas part (iii) is a purely algebraic result that can be applied more broadly. In fact, Lemma \ref{6.2}(iii) is well known; see e.g. \cite[Chapter 3, Lemma 1.1]{Passman}. We include a short proof for the convenience of the reader.

\begin{lemma} \label{6.2}
Let $G$ be a group with finite generating set $X$, and denote word-length with respect to $X$ by $\vert \cdot \vert$. Let $\omega$ be a radial weight on $G$, and denote by $\tau_n$ the value that $\omega$ takes on $S_n$. Assume that $(\tau_n)$ is tail-preserving, and let $D>0$ be a constant as in \eqref{eq5.1}. Consider $\C G \subset \lo(G, \omega)$.
\begin{enumerate}
\item[\rm (i)] Let $u\in G$ be expressed as $u = y_1 \cdots y_n$ for $y_1,\ldots,y_n \in X$, where $n= \vert u \vert$. Then
$$\delta_e - \delta_u = \sum_{x \in X} f_x*(\delta_e - \delta_x)$$
for some $f_x \in \C G \: (x \in X)$ each of which may be taken to have the form
$$f_x = \sum_{j=0}^{n-1} a^{(j)}_x,$$ 
where each $a_x^{(j)}$ is either $0$ or $\delta_{y_1 \ldots y_j}$ in the case that $j \neq 0$, and either $0$ or $\delta_e$ in the case that $j=0$.
\item[\rm (ii)] Each $f_x$ in {\rm (i)} satisfies
\begin{equation} \label{eq6.2a} 
\Vert f_x \Vert \leq \frac{1}{D}\omega(u) \quad (x \in X).
\end{equation}

\item[\rm (iii)] As a left ideal in $\C G$, $\C_0G$ is generated by the elements 
$$\delta_e - \delta_x \quad (x \in X).$$
\end{enumerate}
\end{lemma}

\begin{proof}
(i) We proceed by induction on $n = \vert u \vert$. The case $n = 1$ is trivial, so suppose that $ n > 1$. Set $v=y_2 \cdots y_n \in S_{n-1}$. By the induction hypothesis applied to $v$,
$$\delta_e - \delta_v = \sum_{x \in X} g_x*(\delta_e - \delta_x),$$
where
$$g_x = \sum_{j=0}^{n-2}b_x^{(j)} \quad (x \in X)$$
and each $b_x^{(j)}$ is either $0$, $\delta_{y_2 \cdots y_{j+1}}$ or $\delta_e$. We have
\begin{align*}
\delta_e - \delta_u &= \delta_{y_1}*(\delta_e - \delta_v) + (\delta_e - \delta_{y_1}) 
= \delta_{y_1}* \sum_{x \in X}g_x*(\delta_e - \delta_x) + (\delta_e - \delta_{y_1}) \\
&= \sum_{x \neq y_1} \delta_{y_1}*g_x*(\delta_e - \delta_x) +(\delta_{y_1}*g_{y_1}+\delta_e)*(\delta_e - \delta_{y_1}).
\end{align*}
We define 
$$ f_x = 
\begin{cases} 
\delta_{y_1}*g_x  &(x \neq y_1),\\
 \delta_{y_1}*g_{y_1} + \delta_e  &(x=y_1),
\end{cases}$$
and check that each $f_x$ can be written in the required form. To see this, set
$$ a_x^{(j)} = 
\begin{cases} 
\delta_{y_1}*b_x^{(j-1)} & (j=1,\ldots,n-1),\\
0   & (x \neq y_1, \: j=0),\\

\delta_e & (x= y_1, \: j=0).
\end{cases}$$
It is easily checked that each  $a^{(j)}_x$ has the required form, and that $f_x = \sum_{j=0}^{n-1} a_x^{(j)} \ (x \in X)$. This completes the induction.

(ii) Using part (i), we see that 
$$\Vert f_x \Vert  = \sum_{j=0}^{n-1} \Vert a^{(j)}_x \Vert \quad (x \in X)$$
and, since, for each $x \in X$, every non-zero $a_x^{(j)}$ is $\delta_w$ for some $w \in S_j$, we have
$$\Vert f_x \Vert \leq \sum_{j=0}^{n-1} \tau_j \leq \frac{1}{D} \tau_n = \frac{1}{D}\omega(u),$$
as required.

(iii) Let $\sum \limits_{i=0}^{N} \alpha_i \delta_{u_i} \in \C_0 G$. A simple calculation shows that
$$ \sum \limits_{i=0}^{N} \alpha_i \delta_{u_i}  = \sum \limits_{i=0}^N \left(\sum \limits_{j=0}^i \alpha_i \right)(\delta_{u_i} - \delta_{u_{i+1}}).$$
Moreover, for each $i \in \N$, we have $\delta_{u_i} - \delta_{u_{i+1}} = (\delta_e-\delta_{u_{i+1}})-(\delta_e-\delta_{u_i})$, so that the result follows from (i).
\end{proof}

By analogy to our approach in Section 4, we associate to each function $f\in \lo(G, \omega)$, a complex-valued sequence $(\sigma_n(f))$, defined by

\begin{equation} \label{eq6.2a} 
\sigma_n(f) = \sum \limits_{u \in B_n} f(u).
\end{equation}

\begin{lemma} \label{6.4} 
Let $G$ be a group generated by a finite, symmetric set $X$, and let $\omega$ be a weight on $G$. Let $\tau = (\tau_n)$ be defined by {\rm (\ref{eq6.1})}.
Let $g \in \lo(G, \omega)$ and $x \in X$, and write 
$$ \sigma_n = \sigma_n[g*(\delta_e -  \delta_x)] \quad (n \in \N).$$
Then $(\sigma_n) \in \lo(\tau)$.
\end{lemma}

\begin{proof}
Begin by repeating exactly the argument from the beginning of the proof of Lemma \ref{4.4}, to obtain
$$\sigma_n = \sum \limits_{u \in S_n}g(u) - \sum \limits_{u \in B_n x^{-1} \setminus B_{n-1}} g(u). $$
Again we have  $B_n x^{-1} \setminus B_{n-1} \subseteq  S_n \cup S_{n+1}$. Taking $C$ as in \eqref{eq6.2}, we compute
\begin{align*}
\tau_n \vert \sigma_n \vert &\leq 2\sum \limits_{u \in S_n} \vert g(u) \vert \tau_n + \sum \limits_{u \in S_{n+1}} \vert g(u) \vert \tau_n \\
&\leq  2\sum \limits_{u \in S_n} \vert g(u) \vert \tau_n + C\sum \limits_{u \in S_{n+1}} \vert g(u) \vert \tau_{n+1} \\
&\leq   2\sum \limits_{u \in S_n} \vert g(u) \vert \omega(u) + C\sum \limits_{u \in S_{n+1}} \vert g(u) \vert \omega(u), \\
\end{align*}
where we have used Lemma \ref{6.1} in the second line, and \eqref{eq6.1} in the third line. Since the sets $S_n$ are pairwise disjoint,
we conclude that
$$\sum \limits_{n=1}^{\infty} \tau_n \vert \sigma_n \vert \leq (2+C) \Vert g \Vert_\omega < \infty.$$
Hence $(\sigma_n)  \in \lo(\tau)$, as claimed.
\end{proof}

The following gives a strategy for showing that $\lo_{\,0}(G, \omega)$ fails to be finitely-generated, for finitely-generated groups $G$ and certain weights $\omega$ on $G$.

\begin{theorem} \label{6.5} 
Let $G$ be an infinite group generated by the finite, symmetric set $X$, and let $\omega$ be a weight on $G$. Let $\tau = (\tau_n)$ be defined by {\rm (\ref{eq6.1})}. Suppose that $\lo_{\,0}(G, \omega)$ is finitely-generated. Then $\tau$ is tail-preserving.
\end{theorem}

\begin{proof}
By Lemmas \ref{2.1} and \ref{2.2}, we may suppose that $\lo_{\,0}(G, \omega)$ is generated by a finite subset of $\C_0 G$, and hence, by Lemma \ref{6.2}(iii), we may suppose that each generator has the form $\delta_e - \delta_x$, for some $x \in X$. Therefore every element of $\lo_{\, 0}(G, \omega)$ is a finite linear combination of elements of the form $g*(\delta_e - \delta_x)$, where $g \in \lo(G, \omega)$ and $x \in X$, and so, by Lemma \ref{6.4}, $(\sigma_n(f)) \in \lo(\tau)$ for every $ f \in \lo_{\,0}(G, \omega)$. 

Assume for contradiction that $\tau$ fails to be tail-preserving. Then there exists a sequence $(\alpha_n)$ of non-negative reals such that $(\alpha_n) \in \lo(\tau)$, but  such that $\left ( \sum \limits_{j= n+1}^\infty \alpha_j \right ) \notin \lo(\tau)$. For $n \in \N$, let $y_n \in S_n$ satisfy $\omega(y_n) = \tau_n,$
and define
$$f = \zeta \delta_e - \sum \limits_{j=1}^\infty \alpha_n\delta_{y_n},$$
where $\zeta = \sum \limits_{j=1}^\infty \alpha_j$. Then $f$ is well defined, because $(\alpha_n) \in \lo(\tau)$, and clearly $\varphi_0(f) = 0$. However,
$$\sigma_n(f) = \zeta - \sum \limits_{j=1}^n \alpha_j  = \sum_{j=n+1}^\infty \alpha_j ,$$
so that $(\sigma_n(f)) \notin \lo(\tau)$ by the choice of $\alpha$, contradicting Lemma \ref{6.4}.
\end{proof}

We are now ready to prove Theorem \ref{1.6}, which completely characterizes finite generation of the augmentation ideal in the case where the weight is radial. In particular, this characterization establishes the Dales--\.Zelazko conjecture for $\lo(G, \omega)$ for many groups $G$ and weights $\omega$.

\begin{proof}[Proof of Theorem \ref{1.6}]
If $\lo_{\,0}(G, \omega)$ is finitely-generated then, by Theorem \ref{6.5}, $(\tau_n)$ is tail-preserving .

Suppose that $(\tau_n)$ is tail-preserving.
Write $X = \lbrace x_1,\ldots,x_r \rbrace$ and enumerate $G$ as $G= \lbrace u_0=e, u_1, u_2, \ldots \rbrace$. Let $f = \sum_{n=0}^\infty \alpha_n \delta_{u_n} \in \lo_{\,0}(G, \omega)$, and let $D>0$ be as in \eqref{eq5.1}. By Lemma \ref{6.2}, for each $n \in \N$, there exist $g_n^{(1)},\ldots,g_n^{(r)} \in \C G$ such that $\delta_e-\delta_{u_n} = \sum_{i=1}^r g_n^{(i)}*(\delta_e-\delta_{x_i})$ and 
$$\Vert g_n^{(i)} \Vert \leq \frac{1}{D}\Vert \delta_{u_n} \Vert \quad (i=1,\ldots,r).$$
This implies that, for each $i=1,\ldots,r$, we may define an element of $\lo(G, \omega)$ by
$$s^{(i)} = -\sum_{n=1}^\infty \alpha_n g_n^{(i)} .$$
Then 
\begin{align*}
f &= \sum_{n=0}^\infty \alpha_n \delta_{u_n} - \left( \sum_{n=0}^\infty \alpha_n  \right) \delta_e = -\sum_{n=1}^\infty \alpha_n (\delta_e -\delta_{u_n}) \\
&= - \sum_{n=1}^\infty \alpha_n \left( \sum_{i=1}^r g_n^{(i)}*( \delta_e - \delta_{x_i}) \right) 
= \sum_{i=1}^r s^{(i)}*(\delta_{e}-\delta_{x_i}).
\end{align*}
As $f$ was arbitrary, it follows that  $\lo_{\,0}(G, \omega)$ is generated by the elements ${\delta_e -\delta_{x_1}}, \ldots , {\delta_e - \delta_{x_r}}$.
\end{proof}

We now prove Corollary \ref{1.7}, part (ii) of which shows that it can happen that $\lo_{\,0}(G, \omega)$ is finitely-generated, for certain infinite groups $G$ and certain weights $\omega$.

\begin{proof}[Proof of Corollary \ref{1.7}]
(i) Lemma \ref{5.3}(i) implies that, for such a weight, the sequence $(\tau_n)$ defined in Theorem \ref{1.6} is not tail-preserving, and the result follows from that theorem.  

(ii) By Theorem \ref{5.2}, the sequence $(\tau_n)$ of Theorem \ref{1.6} is tail-preserving.
\end{proof}

Let $G$ be a discrete group, and $G'$ its commutator subgroup. We conclude this section by remarking that, if $[G:G'] = \infty$, then $\lo(G, \omega)$ safisfies the Dales--\.Zelazko conjecture for every weight  $\omega$. The reasoning is as follows. By \cite[Corollary 1.7]{ft1978}, the conjecture holds for $\lo(H, \omega)$ whenever $H$ is an abelian group and $\omega$ is a weight on $H$. Then, by \cite[Theorem 3.1.13]{RS}, given $G$ and $\omega$, there exists a weight $\widetilde{\omega}$ on $G/G'$ such that $\lo(G/G', \widetilde{\omega})$ is a quotient of $\lo(G, \omega)$. Finally, by the commutative result, there is some maximal ideal in  $\lo(G/G', \widetilde{\omega})$ which is not finitely-generated, and taking its preimage under the quotient map gives a maximal left ideal in $\lo(G, \omega)$ which is not finitely-generated. However, we have not been able to establish the Dales--\.Zelazko conjecture for an arbitrary weighted group algebra.

\section{Examples on $\Z$ and $\Z^+$}
\noindent
In this section we look at some specific examples of weighted algebras on $\Z$ and $\Z^+$, and consider how they fit into the more general theory of maximal ideals in commutative Banach algebras. When convenient, we shall sometimes write $\omega_n$ in place of $\omega(n)$.

For a commutative Banach algebra $A$ we shall denote the character space of $A$ by $\Phi_A$, and for an element $a \in A$, we shall denote by $\widehat{a}$ its Gelfand transform. We first recall Gleason's Theorem \cite[Theorem 15.2]{Stout1971}:

\begin{theorem} \label{7.1a}
Let $A$ be a commutative Banach algebra, with unit $1$ and take $\varphi_0 \in \Phi_A$. Suppose that $\ker \varphi_0$ is finitely-generated by $g_1,\ldots,g_n$, and take $\gamma : \Phi_A \rightarrow \C^n$ to be the map given by
$$\gamma( \varphi) = (\varphi(g_1),\ldots, \varphi(g_n)) \quad (\varphi \in \Phi_A).$$
Then there is a neighbourhood $\Omega$ of $0$ in $\C^n$ such that:
\begin{enumerate}
\item[\rm (i)] $\gamma$ is a homeomorphism of $\gamma^{-1}(\Omega)$  onto an analytic variety $E$ of $\Omega$;
\item[\rm (ii)] for every $a \in A$, there is a holomorphic function $F$ on $\Omega$ such that $\widehat{a} = F \circ \gamma$ on $\gamma^{-1}(\Omega)$;
\item[\rm (iii)] if $\varphi \in \gamma^{-1}(\Omega)$, then $\ker \varphi$ is finitely-generated by
$$g_1 - \varphi (g_1)1,\ldots, g_n - \varphi(g_n)1.$$
\end{enumerate}
\end{theorem}

It is natural to wonder whether there are circumstances under which a converse holds. For instance, suppose we have a commutative Banach algebra $A$ such that there is an open subset $U$ of the character space, which is homeomorphic to an open subset of $\C^n$, and such that $\hat{a}$ is holomorphic on $U$ under this identification for every $a \in A$. Does it then follow that the maximal ideals corresponding to points of $U$ are finitely-generated? T. T. Reed gave an example \cite[Example 15.9]{Stout1971} which shows that this need not be true in general, even for uniform algebras. We note that the character space in Reed's example is very complicated. In this section we give two examples of commutative Banach algebras for which the converse to Gleason's Theorem fails to hold, and whose character spaces are the disc and the annulus respectively. The first (Theorem \ref{7.5}) shows that there is no general converse to Gleason's Theorem for the class of natural Banach function algebras on simply connected compact plane sets. The second (Theorem \ref{7.7}) shows that there is no general converse to Gleason's Theorem for the class of weighted abelian group algebras. Interestingly, these examples rely on constructing counterparts to the sequence $(\tau_n)$ of Lemma \ref{5.3}(ii) satisfying the additional constraints that the sequence must now be a weight on $\Z^+$ in Theorem \ref{7.5}, and a weight on $\Z^+$ admitting an extension to $\Z$ in Theorem \ref{7.7}.

We note that many authors have considered similar questions for the algebras $A(\Omega)$ and $H^\infty(\Omega)$, for $\Omega \subset \C^n$ a domain, and this is sometimes referred to as Gleason's problem; see e.g. \cite{Kot2002}, \cite{LW2002}.

Before we construct our examples, we first recall some facts about weights on $\Z$ and $\Z^+$; see \cite[Section 4.6]{D} for more details.

Let $\omega$ be a weight on $\Z$. The character space of $\lo(\Z, \omega)$ may be identified with the annulus ${\lbrace z \in \C : \rho_1 \leq \vert z \vert \leq \rho_2 \rbrace},$ where 
$$\rho_1 = \lim \limits_{n \rightarrow \infty} \omega_{-n}^{-1/n}  \qquad {\rm and} \qquad \rho_2 = \lim \limits_{n \rightarrow \infty} \omega_n^{1/n}.$$
The identification is given  by $\varphi \mapsto \varphi(\delta_1)$, for $\varphi$ a character. Note that $\rho_1 \leq 1 \leq \rho_2$. As it is easily seen to be semi-simple, $\lo(\Z, \omega)$ may be thought as a Banach function algebra on the annulus, and in fact these functions are all holomorphic on the interior of the annulus. We denote by $M_z$ the maximal ideal corresponding to the point $z$ of the annulus, and observe that the augmentation ideal is $M_1$.

Now instead let $\omega$ be a weight  on $\Z^+$. The situation for $\lo(\Z^+, \omega)$ is analogous to the situation above. Now the character space is identified with the disc ${\lbrace z \in \C: \vert z \vert \leq \rho \rbrace}$, where $\rho = \lim \limits_{n \rightarrow \infty} \omega_{n}^{1/n}$, and $\lo(\Z^+, \omega)$ may be considered as a Banach function algebra on this set, with the property that each of its elements is holomorphic on the interior. In this context $M_z$ denotes the maximal ideal corresponding to the point $z$ of the disc.

Before giving our examples we characterize those weights for which $\lo_{\,0}(\Z, \omega)$ is finitely-generated. Note that this is a slight improvement, for the  group $\Z$, on Theorem \ref{1.6} since we no longer need to assume that the weight is radial.

\begin{theorem} \label{7.1}
Let $\omega$ be a weight on $\Z$. Then $\lo_{\,0}(\Z, \omega)$ is finitely-generated if and only if both sequences $\left (\omega_n \right )_{n \in \N}$ and $\left (\omega_{-n} \right )_{n \in \N}$ are tail-preserving. 
\end{theorem}

\begin{proof}
Set $A = \lo(\Z, \omega)$. Suppose that $\left (\omega_n \right )_{n \in \N}$ is not tail-preserving. Then we can repeat the proof of Theorem \ref{6.5} with $G= \Z$ essentially unchanged, except that now we insist that all functions appearing in it have support contained in $\Z^+$, to show that $\lo_{\,0}(\Z, \omega)$ is not finitely-generated. By symmetry, the same conclusion holds if instead $\left (\omega_{-n} \right )_{n \in \N}$ fails to be tail-preserving.

Now suppose that  $\left (\omega_n \right )_{n \in \N}$ and $\left (\omega_{-n} \right )_{n \in \N}$ are both tail-preserving. Let $f \in \lo_{\,0}(\Z, \omega)$, and suppose for the moment that $\supp f \subset \Z^+$. Then we have
\begin{align*}
\sum \limits_{n=0}^\infty \omega_n \left \vert \sum \limits_{i = 1}^n f(i) \right \vert
= \sum \limits_{n=0}^\infty \omega_n \left \vert \sum \limits_{i = n+1}^\infty f(i) \right \vert < \infty,
\end{align*}
and so we may define $g \in A$ by 
$$g = -\sum \limits_{n=0}^\infty \left ( \sum \limits_{i = 0}^n f(i) \right )\delta_n.$$
Then
\begin{align*}
g*(\delta_1 - \delta_0) &=  -\sum \limits_{n=0}^\infty \left (\sum \limits_{i=0}^n f(i) \right) \left(\delta_{n+1} - \delta_{n} \right)\\
&= \sum \limits_{n=0}^\infty \left ( \sum \limits_{i=0}^n f(i) - \sum \limits_{i=0}^{n-1} f(i) \right) \delta_n 
= \sum \limits_{n=0}^\infty f(n)\delta_n =  f.
\end{align*}
Hence
$$f = g*(\delta_1 - \delta_0) \in A*(\delta_1 - \delta_0).$$
A similar argument shows that, if $\supp f \subset \Z^-$, then 
$$f \in A*(\delta_{-1} - \delta_0).$$
But any $f \in \lo_{\, 0}(\Z, \omega)$ can be written as $f = f_1+f_2$ for $f_1,f_2 \in \lo_{\,0}(\Z, \omega)$, with $\supp f_1 \subset \Z^+$ and $\supp f_2 \subset \Z^-$, and so we see that 
$$\lo_{\,0}(\Z, \omega) = A*(\delta_1 - \delta_0) +A*(\delta_{-1} - \delta_0)$$
is finitely-generated, as required
\end{proof}

We now construct the first of our special weights described at the beginning of the section. This is a weight on $\Z^+$ such that neither the augmentation ideal nor $M_0$ are finitely-generated.

\begin{lemma} \label{7.4}
Let $\rho >1$. Then there exists a weight $\omega$ on $\Z^+$, satisfying $\lim \limits_{n \rightarrow \infty} \omega_n^{1/n} = \rho$ such that there exists a strictly increasing sequence of natural numbers $(n_k)$ with 
\begin{equation} \label{eq7.1}
\frac{\omega_{n_k+1}}{\omega_{n_k}} \leq \frac{\rho+1}{k} \quad (k \in \N).
\end{equation}
\end{lemma}

\begin{proof}
First, we define inductively a non-increasing null sequence  $(\eps(n))$ of positive reals, as follows.
Set $\eps(0) = 1$. Since $\lim \limits_{n \rightarrow \infty} (1/n)^{1/n} =1$, we can find an integer $n_1$ such that 
$$0<\left( \frac{1}{n_1} \right)^{1/n_1}(\rho +1)-\rho .$$
Define $\eps(n) = 1$ for $n \leq n_1$, and then choose $\eps (n_1 +1)$ such that
$$0< \eps(n_1+1) < \left( \dfrac{1}{n_1} \right)^{1/n_1}(\rho +1)-\rho.$$
Note also that $\eps(n_1+1) < 1$.

Now take $k \geq 2$, and suppose that we have already defined a strictly increasing sequence of integers $n_1,\ldots,n_{k-1}$, and defined $\eps(n)$ for $n \leq n_{k-1} +1$. Then choose $n_k \in \N$, with $n_k > n_{k-1}$ and such that 
$$0<\left( \frac{1}{n_k} \right)^{1/n_k}(\rho +\eps(n_{k-1}+1))-\rho.$$
Define $\eps(n) = \eps(n_{k-1}+1)$ for $n_{k-1}+1 < n \leq n_k$, and then choose $\eps (n_k+1)$ such that
$$0< \eps(n_k+1) <  \left( \frac{1}{n_k} \right)^{1/n_k}(\rho +\eps(n_{k-1}+1))-\rho,$$
whilst ensuring that $\eps(n_k+1) < \min \lbrace 1/k, \eps(n_k) \rbrace$. This completes the inductive construction of $\eps$.

Now define 
$$\omega_n = (\rho+\eps(n))^n \quad (n \in \Z^+).$$
Then $\omega := (\omega_n)$ is a weight on $\Z^+$, because
\begin{align*}
\omega_{m+n} &= (\rho+\eps(m+n))^m(\rho+\eps(m+n))^n \\
 &\leq (\rho+\eps(m))^m(\rho+\eps(n))^n = \omega_m \omega_n \quad (m,n \in \Z^+),
\end{align*}
where we have  used the fact that $\eps$ is non-increasing. As $\lim \limits_{n \rightarrow \infty} \eps(n) = 0$, we have $\lim \limits_{n  \rightarrow \infty} \omega_n^{1/n} = \rho$. It remains to show that \eqref{eq7.1} holds.

For $k \in \N$, we have
$$\frac{\omega_{n_k+1}}{\omega_{n_k}} \leq (\rho+1)\frac{(\rho +\eps(n_k+1))^{n_k}}{(\rho +\eps(n_k))^{n_k}}.$$
However
$$\rho +\eps(n_k+1) < \left( \dfrac{1}{n_k} \right)^{1/n_k}(\rho + \eps(n_{k-1}+1)) = \left( \dfrac{1}{n_k} \right)^{1/n_k}(\rho+\eps(n_k)),$$
which implies that 
$$\dfrac{(\rho +\eps(n_k+1))^{n_k}}{(\rho +\eps(n_k))^{n_k}} < \dfrac{1}{n_k} \leq \frac{1}{k},$$
and \eqref{eq7.1} now follows.
\end{proof}

As $\lim \limits_{n \rightarrow \infty} \omega_n^{1/n} = \inf \limits_{n \in  \N} \omega_n^{1/n}$ by \cite[Proposition A.1.26(iii)]{D}, the weight constructed in Theorem \ref{7.4} satisfies $\omega_n \geq \rho^n \: (n \in \N)$. However, Lemma \ref{5.2} implies that $\omega$ is not tail-preserving, as
$$\liminf \limits_n \left( \omega_{n+1} \left( \sum_{j=1}^{n} \omega_j \right)^{-1} \right) \leq \liminf \limits_n \frac{\omega_{n+1}}{\omega_n} =0.$$
Hence we have a version of Lemma \ref{5.3}(ii) in which the sequence is also a weight.

\begin{theorem} \label{7.5} 
Let $\omega$ denote the weight constructed in Lemma {\rm \ref{7.4}}. Then neither $M_1$ nor $M_0$ is finitely-generated, even though both 0 and 1 correspond to interior points of the character space.
\end{theorem}

\begin{proof} 
Set $A = \ell^{\, 1}(\Z^+, \omega)$ and assume towards a contradiction that $M_0$ is finitely-generated. Note that $M_0  = \lbrace f \in A: f(0) = 0 \rbrace$, so that every finitely supported element of $M_0$ is of the form $g*\delta_1$, for some $g \in A$. By Lemmas \ref{2.1} and \ref{2.2}, we may suppose that the generators of $M_0$ have finite support, and as they also lie in $M_0$, we may factor out a $\delta_1$ from each one. It follows that $M_0 = A*\delta_1$. Define a sequence of non-negative reals by 
$$\alpha_j = 
\begin{cases}
 (k \omega_{n_k})^{-1} & \quad \text{if } j = n_k+1 ,\\
 0 & \quad \text{otherwise.}
\end{cases}$$
Let $f = \sum \limits_{j = 1}^\infty \alpha_j \delta_j$. Then by \eqref{eq7.1} we have

\begin{align*} 
\sum \limits_{j = 0}^\infty \vert \alpha_j \vert \omega_j = \sum \limits_{k = 1}^\infty \frac{\omega_{n_k +1}}{k \omega_{n_{k}}} 
\leq (\rho +1) \sum \limits_{k = 1}^\infty \frac{1}{k^2} < \infty.
\end{align*}
This shows that $f \in A$, and so clearly $f \in M_0$. Assume that $f = g* \delta_1$ for some $g \in A$. Then $g$ must satisfy $g(j-1) = f(j) \; (j \in \N)$. However,
$$\sum \limits_{j = 1}^\infty \vert f(j) \vert \omega_{j-1} = \sum \limits_{k = 1}^\infty \frac{1}{k}= \infty,$$
so that $g \notin A$.

The case of $M_1$ is very similar. This time we know that, if $M_1$ is finitely-generated, it must equal $A*(\delta_0 - \delta_1)$. By the remark preceding the theorem $\omega$ is not tail-preserving, and so there exists some sequence $(\alpha_n) \in \lo(\omega)$, such that 
$$\sum_{n=1}^\infty \omega_n \left \vert \sum_{j=n+1}^\infty \alpha_j \right \vert = \infty.$$
 Let $\zeta = \sum_{n=1}^\infty \alpha_n$, and let $f = \zeta \delta_0 - \sum_{n=1}^{\infty} \alpha_n \delta_n$. Then $f \in M_1$. Assume that $f = g*(\delta_0-\delta_1)$ for some $g \in A$. A short calculation implies that
$$g(n) = \sum_{j=0}^n f(j) = -\sum_{j=n+1}^\infty \alpha_j$$
 for all $n \geq 1$, contradicting the fact that $\sum_{n=0}^\infty \omega_n \vert g(n) \vert < \infty$.
\end{proof}

We remark that a weight $\omega$ on $\Z^+$ extends to a weight  on $\Z$ if and only if $\sup \limits_{n \in \N} \omega_n/\omega_{n+1} < \infty$. The  ``only if'' direction of this implication just follows from submultiplicativity of the weight at $-1$. For the ``if'' direction, set $C = \sup \limits_{n \in \N} \omega_n/\omega_{n+1}$. Then it is routine to verify that $\omega_{-n} = C^n \omega_n \: (n \in \N)$ defines an extension. It follows from this observation that the weight constructed in Lemma \ref{7.4} admits no extension to $\Z$. However, a different construction does allow us to do something similar on $\Z$.
 
\begin{lemma} \label{7.6} 
Let $\rho>1$. Then there exists a weight $\omega$ on $\Z$ satisfying $\lim \limits_{n \rightarrow \infty} \omega_n ^{1/n} = \rho$ and $\lim \limits_{n \rightarrow \infty} \omega_{-n}^{-1/n} < 1$, but such that $(\omega_n)_{n \in \N}$ is not tail-preserving.
\end{lemma}

\begin{proof}
 With the preceding remark in mind, we construct first a weight $\gamma$ on $\Z^+$ satisfying $\sup \limits_{n \in \N} \lbrace \gamma_n/ \gamma_{n+1} \rbrace \leq \rho+1$, which ensures that $\gamma$ extends to  a weight on $\Z$. In the end we shall define $\omega$ by $\omega_n = \rho^{\vert n \vert}\gamma_n$.

We set $n_k = 2^k -1  \; (k \in \N)$. We define $\gamma$ on $\lbrace 0,1,2,3 \rbrace$ by 
$$\gamma_0 = 1, \; \gamma_1 = \rho+1, \; \gamma_2 = (\rho+1)^2, \; \gamma_3 = \rho+1.$$
We then recursively define
$$\gamma_j = (\rho+1) \gamma_{j - n_k} \quad (n_k \leq j < n_{k+1}, \: k \geq 2).$$

We observe that
\begin{equation} 
\gamma_{n_k - i} = (\rho+1)^{i+1} \quad (0 \leq i \leq k-1, \: k \geq 2).
\label{star}
\end{equation}
This follows by an easy induction on $k$. Indeed, the base case can be seen to hold by inspection, and for $k \geq 3$ we see that
\begin{align*}
\gamma_{n_k -i} &= (\rho+1) \gamma_{n_k -i - n_{k-1}} = (\rho+1) \gamma_{n_{k-1} +1 -i} \\
&= (\rho+1)(\rho+1)^i = (\rho+1)^{i+1} \quad (1 \leq i \leq k-1)
\end{align*}
and
$$\gamma_{n_k} = (\rho+1)(\gamma_{n_k - n_k}) = (\rho+1)\gamma_0 = (\rho+1).$$

We now claim that 
\begin{equation} \label{dagger}
\gamma_j \leq (\rho+1) \gamma_{j+1} \quad (j \in \N). 
\end{equation}
Again, this  can be seen by inspection for $j \leq n_2$, and we then proceed by induction on $k$. Indeed, if $j \in [n_k, n_{k+1}-2]$ then
$$\frac{\gamma_j}{\gamma_{j+1}} = \frac{(\rho+1)\gamma_{j-n_k}}{(\rho+1)\gamma_{j+1 - n_k}} = \frac{\gamma_{j-n_k}}{\gamma_{j+1 - n_k}} \leq (\rho+1).$$
When $j  = n_{k+1}-1$ and $j+1 = n_{k+1}$, then, by (\ref{star}), we have
$$\frac{\gamma_j}{\gamma_{j+1}} = \frac{(\rho+1)^2}{(\rho+1)} = \rho+1,$$
establishing the claim.

Now we are ready to prove that $\gamma$ really is a weight. That $\gamma$ is submultiplicative on $\lbrace 0,1,2,3 \rbrace$ can be seen by inspection. Let $i, j \in \N$, with  $i \leq j$, and let $k \in \N$ satisfy $i+j \in [n_{k+1}, n_{k+2})$. We proceed by induction on $i+j$. If $j<n_k$ then $i+j <2n_k <n_{k+1}$, so we must have $j \geq n_k$. There are three cases. Firstly, if $j \geq n_{k+1}$, then
$$\gamma_{i+j} = (\rho+1) \gamma_{i+j-n_k} \leq (\rho+1) \gamma_i \gamma_{j-n_k} = \gamma_i \gamma_j.$$
If instead $j<n_{k+1}$, but $i \geq n_k$, then
\begin{align*}
\gamma_{i+j} &= (\rho+1) \gamma_{i+j - n_{k+1}} = (\rho+1)\gamma_{(i-n_k)+(j-1-n_k)} \\
&\leq (\rho+1) \gamma_{i-n_k}\gamma_{j-1-n_k} 
=\frac{1}{\rho+1} \gamma_i \gamma_{j-1}
\leq \gamma_i \gamma_j,
\end{align*}
by \eqref{dagger}.
Finally, suppose that $i<n_k$ and $n_k \leq j <n_{k+1}$. In this case we have $i+j-2^k < n_{k+1},$
 and, since $i+j \geq n_{k+1},$ we also have $$i+j -2^k \geq n_{k+1} - 2^k = n_k,$$ 
so that $i+j - 2^k \in [n_k, n_{k+1}).$ Then the formula $i+j - n_{k+1} = i+j-(n_k+2^k)$ implies that 
\begin{align*}
\gamma_{i+j} = (\rho+1) \gamma_{i+j-n_{k+1}} = (\rho+1)\gamma_{i+j-2^k - n_k} 
= \gamma_{i+j-2^k}.
\end{align*}
Therefore
\begin{align*}
\gamma_{i+j} &= \gamma_{i+j-2^k} \leq \gamma_{i-1}\gamma_{j+1-2^k} = \gamma_{i-1}\gamma_{j-n_k} 
\leq (\rho+1)\gamma_i \gamma_{j-n_k} \quad \text{by \eqref{dagger}} \\
&= \gamma_i \gamma_j.
\end{align*}
This concludes the proof that $\gamma$ is a weight. By (\ref{dagger}), it extends to a weight  on $\Z$, which we also denote by $\gamma$.

Define $\omega = (\omega_n)$ by $\omega_n = \rho^{\vert n \vert} \gamma_n \; (n \in \Z)$. As $\gamma_{n_k}^{1/n_k} = (\rho+1)^{1/n_k}$ for all $k \geq 2$, we must have $\lim \limits_{n \rightarrow \infty} \gamma_n^{1/n} = 1$, and hence $\lim \limits_{n \rightarrow \infty} \omega_n^{1/n} = \rho$. Furthermore,
$$\lim \limits_{n \rightarrow \infty} \omega_{-n}^{-1/n} = \frac{1}{\rho} \lim \limits_{n \rightarrow \infty} \gamma_{-n}^{-1/n} \leq \frac{1}{\rho} <1,$$
as required.

It remains to show that $(\omega_n)_{n \in \N}$ is not tail-preserving. We compute
\begin{align*}
\frac{\omega_{n_k}}{\sum_{j=1}^{n_k-1} \omega_j} \leq \frac{\omega_{n_k}}{\omega_{n_k -(k-1)}} = \frac{\rho^{n_k}(\rho+1)}{\rho^{n_k+1-k}(\rho+1)^k} 
= \left(\frac{\rho}{\rho+1} \right)^{k-1},
\end{align*}
which tends to 0 as $k$ goes to infinity. In particular, $(\omega_n)_{n \in \N}$ violates \eqref{eq5.1}, so it is not tail-preserving. 
\end{proof}

\begin{theorem} \label{7.7}
Let $\omega$ be the weight constructed in Lemma {\rm \ref{7.6}}. Then the augmentation ideal $\lo_{\,0}(\Z, \omega)$ fails to be finitely-generated, despite corresponding to an interior point of the character space.
\end{theorem}

\begin{proof}
By construction $M_1$ corresponds to an interior point of the annulus. Now apply Theorem \ref{7.1}.
\end{proof}

\subsection*{Acknowledgements}
\noindent 
I would like to thank my supervisor Garth Dales for suggesting the topic of this paper, and for many helpful discussions whilst I was conducting this research. I would also like to thank my second supervisor Niels Laustsen, as well as Yemon Choi, for patiently listening to various expositions of my work when in progress, and for giving constructive and encouraging feedback. Finally, I would like to thank the referee for his/her valuable comments.
 




\normalsize
\baselineskip=17pt


\end{document}